\documentclass[12pt]{amsart}

\usepackage[margin=2.5cm]{geometry}
\usepackage{amssymb}
\usepackage{latexsym}
\usepackage{enumerate}

\usepackage{color}

\newtheorem{theorem}{Theorem}[section]
\newtheorem{lemma}[theorem]{Lemma}
\newtheorem{proposition}[theorem]{Proposition}

\theoremstyle{definition}
\newtheorem{definition}[theorem]{Definition}

\theoremstyle{remark}
\newtheorem{remark}[theorem]{Remark}

 \numberwithin{equation}{section}

\newcommand{\hatbox}{\widehat{\square}}
\newcommand{\supp}{\operatorname{supp}}

\renewcommand{\Im}{\operatorname{Im}}

\newcommand{\MF}{\mathcal{F}}
\newcommand{\C}{\mathbb{C}}
\newcommand{\N}{\mathbb{N}}
\newcommand{\R}{\mathbb{R}}

\begin{document}

\title[Marcinkiewicz multipliers]
{Marcinkiewicz multipliers associated with the Kohn  Laplacian on the Shilov boundary\\ of the product domain in  $\C ^{2n}$}

\author{Peng {Chen}}
\address{Department of Mathematics, Sun Yat-Sen University, Guangzhou, P. R. China}
\email{achenpeng@163.com}

\author{Michael G. {Cowling}}
\address{School of Mathematics and Statistics, University of New South Wales, Sydney 2052, Australia}
\email{m.cowling@unsw.edu.au}

\author{Guorong {Hu}}
\address{College of Mathematics and Information Science,
Jiangxi Normal University, Nanchang, Jiangxi 330022, P. R. China}
\email{hugr1984@163.com}

\author{Ji {Li}}
\address{Department of Mathematics, Macquarie University, Sydney, NSW 2109, Australia}
\email{ji.li@mq.edu.au}

\subjclass[2010]{32A50, 32A55, 32T15, 32W30, 32W10}

\keywords{Nonisotropic smoothing operator of order zero, product singular integral of Journ\'e type, Marcinkiewicz multiplier, Kohn  Laplacian}

\date{\today}

\begin{abstract}
Let $M^{(k)}$, $k=1,2,\ldots, n$, be the boundary of an unbounded polynomial domain $\Omega^{(k)}$ of finite type in $\C ^2$, and
let $\Box_b^{(k)}$ be the Kohn  Laplacian on $M^{(k)}$. In this paper,
we study multivariable spectral multipliers $m(\Box_b^{(1)},\ldots, \Box_b^{(n)})$
acting on the Shilov boundary $\widetilde{M}=M^{(1)} \times\cdots\times M^{(n)}$ of the product domain $\Omega^{(1)}\times\cdots\times \Omega^{(n)}$.
We show that if a function $F(\lambda_1, \ldots ,\lambda_n)$ satisfies a Marcinkiewicz-type
differential condition, then the spectral multiplier operator $m(\Box_b^{(1)}, \ldots, \Box_b^{(n)})$ is
a product Calder\'on--Zygmund operator of Journ\'e type.
\end{abstract}

\maketitle

\section{Introduction and statement of main result}
\allowdisplaybreaks

The Shilov boundary $\widetilde{M}=M^{(1)} \times\cdots\times M^{(n)}$ of a product domain $\Omega^{(1)}\times\cdots\times \Omega^{(n)}$ in $\C^{2n}$ is an important model in several complex variables that extends the standard case of the bidisc \cite{NS1}.
Here each $M^{(k)}$ is the boundary of an unbounded polynomial domain $\Omega^{(k)}:= \{(z,w)\in\C ^2: \Im(w)>P^{(k)}(z)\}$, where $P^{(k)}$ is a
real, subharmonic, nonharmonic polynomial of degree $m_k$ (see \cite{NS}). In \cite{NS1}, Nagel and Stein built the theory of product singular integrals and Littlewood--Paley theory on $\widetilde{M}$, based on the solution to the initial value problem and the regularity properties of the heat
operator for the Kohn  Laplacian $\Box_b^{(k)}$ on each $M^{(k)}$.
This work leads immediately to optimal estimates for the behaviour of the solution of the Kohn--Laplace equation on the decoupled boundary in $\C^{n+1}$ \cite{NS2}.
It is worth pointing out that the heat kernels of the Kohn  Laplacian $\Box_b$ on the boundary $M$ (here we drop the superscript $^{(k)}$) do not satisfy standard Gaussian upper bounds; indeed, Nagel and M\"uller (unpublished) showed that the kernel of $e^{-s\Box_b}$ satisfies the pointwise upper bound
\begin{equation}\label{e7.2}
\left|
K_{e^{-s\Box_b}} (x, y)\right| \leq \frac{C}{V(x, \rho(x,y))} \exp\left(-c{\rho(x,y)^{2}/ s}\right)
\end{equation}
for some positive constants $C$ and $c$ (see  also  \cite{Str1}).
Here $\rho$ is the control metric on $M$ and $V(x,\delta)$ is the volume of the metric ball $B(x,\delta): = \{y \in M: \rho(x,y) <\delta\}$.
The measure $V$ is doubling, with upper dimension $Q$, and reverse doubling, with lower dimension $4$. See Section 2 for the details.

One of the important types of singular integrals on the Shilov boundary $\widetilde{M}$ is the family of Marcinkiewicz multipliers associated with the Kohn  Laplacians, that is, the operators $m(\Box_b^{(1)}, \ldots, \Box_b^{(n)})$.
The aim of this paper is to investigate the behaviour of these multipliers.
We show that if the function $m(\lambda_1, \ldots ,\lambda_n)$ satisfies a suitable smoothness condition, defined using Sobolev norms, then $m(\Box_b^{(1)}, \ldots, \Box_b^{(n)})$ is
a product Calder\'on--Zygmund operator of Journ\'e type.

To ease the burden of notational complexity, in what follows we consider only the two parameter case, that is, $n=2$.

Marcinkiewicz multipliers in the Euclidean setting provide one of the most important examples of product Calder\'on--Zygmund operators.
It is well-known that such multipliers are bounded on $L^p$ for $1<p<\infty$. To the best of our knowledge, the earliest such result is due to Gundy
and Stein \cite{GS} in 1979; they investigated the operator $m(\Delta_1, \Delta_2)$ with $m: {\R^2}\to {\C}$  a bounded function and $\Delta_i$  the standard Laplace operator on the Euclidean space $\R^{n_i}$, for $i=1$,~2.
They proved that, if $\alpha_0,\beta_0 \in \N$ are sufficiently large, then $m(\Delta_1, \Delta_2)$, initially defined  on $L^2({\R}^{n_1}\times \R^{n_2})$, extends to a bounded operator on $L^p(\R^{n_1}\times \R^{n_2})$ for all $p\in (1, \infty)$ provided the function $m$ satisfies
\begin{equation}\label{Marcinkiewcz}
    \left| \partial_{\lambda_2}^{\beta}
        \partial_{\lambda_1}^{\alpha} m(\lambda_1, \lambda_2) \right|
    \lesssim \lambda_1^{-\alpha}\lambda_2^{-\beta}
\end{equation}
for all $\alpha\leq \alpha_0$ and $\beta\leq\beta_0$.
Their proof uses the pointwise majorization of $m(\Delta_1, \Delta_2)$ by the Littlewood--Paley product $g$ and
$g_{\lambda}$ functions.
Another approach, due to L.K.~Chen~\cite{Ch}, uses $(H^1, L^1)$ estimates for $L^2$-bounded linear operators, where $H^1$ is the product Hardy space
introduced by Gundy and Stein and further studied by Chang and Fefferman.
Chen proved that it suffices to take $\alpha_0$ and $\beta_0$ equal to
when $\lfloor{n_1}/{2}\rfloor+1$ and $\lfloor{n_2}/{2} \rfloor+1$.

Recently, this Marcinkiewicz multiplier was investigated in the abstract setting of spaces of homogeneous type with the Marcinkiewicz multiplier $m(L_1,L_2)$ associated with nonnegative self-adjoint second-order operators whose heat kernels satisfy
Gaussian upper bounds \cite{CDLWY}.
Suppose that $X_1\times X_2$ is the Cartesian
product of measure spaces $X_1$ and $X_2$ and that $L_1$ and
$L_2$ are nonnegative self-adjoint operators acting on the
spaces $L^2(X_1)$ and $L^2(X_2)$.
 Let $E_{L_i}$ be the projection valued measure associated to $L_i$.
There is a unique spectral decomposition~$E$ such that for all Borel
subsets $A\subset \R^2$, $E(A)$ is a projection on
$L^2({X_1\times X_2})$ and $ E(A_1\times
A_2) = E_{L_1}(A_1)\otimes E_{L_2}(A_2)$ for all Borel subsets
$A_1$ and $A_2$ of $\R $.
Hence for a bounded
Borel function $m: \R^2\to {\C}$, one
may define the spectral multiplier operator $m(L_1, L_2)$
acting on the space $L^2({X_1\times X_2})$ by the formula
\begin{equation}\label{e1.1}
m(L_1, L_2)
:= \iint_{\R \times \R} m(\lambda_1, \lambda_2) \, dE(\lambda_1, \lambda_2).
\end{equation}
Clearly $m(L_1, L_2)$ is bounded on $L^2({X_1\times X_2})$.
If the function $F$ satisfies a suitable mixed Sobolev condition with regularity $s_1>n_1/2$ and $s_2>n_2/2$, then $m(L_1, L_2)$ is bounded on $L^p({X_1\times X_2})$ for $1<p<\infty$ (\cite{CDLWY}).
Here $n_i$ is the upper dimension of $X_i$.
The proof uses a ($H^1$, $L^1$) estimate, where the $H^1$ is the product Hardy space associated with $L_1$ and $ L_2$ introduced in \cite{CDLWY2}.

In the single-parameter compact manifold $M$, Street \cite{Str2}
proved that the spectral multiplier $m(\Box_b)$ is an anisotropic smooth (NIS) operator of order zero if $m: [0,\infty) \to \C$ satisfies the Mihlin--H\"{o}rmander condition
$ |(\lambda \partial_\lambda)^am(\lambda)|\leq C_a$  for all $\lambda \in (0, \infty)$ and all  $a\in \N_0$, then $m(\Box_b)$ is bounded on $L^p(M)$ for all $1<p<\infty$.
Moreover, he pointed out that if one is only concerned with $L^p$ boundedness, then a sufficient condition is
\[
\sup_{t >0} \|\eta(\cdot) m(t\cdot)\|_{ W^{\sigma, 2}(\R )} < \infty
\]
where $\sigma> {Q+1 / 2}$; here $\eta$ is a nontrivial cut-off function on $(0, \infty)$ and $W^{\sigma, 2}(\R )$ is the usual fractional Sobolev space on $\R $. The expected result for $m(\Box_b)$ in the unbounded polynomial domain is also discussed in \cite[Section 8]{Str2}.

In light of \cite{Str1},  it is natural to expect that the Marcinkiewicz multiplier $m(\Box_b^{(1)}, \Box_b^{(2)})$
is a product NIS operator of order zero when $m$ satisfies the decay condition for infinitely many partial derivatives.
Here we provide a more general characterisation.
We aim to prove that when $m$ satisfies a suitable Sobolev estimate, $m(\Box_b^{(1)}, \Box_b^{(2)})$ is a product Calder\'on--Zygmund operator of Journ\'e type \cite{J}.
In this metric space setting, Journ\'e-type product Calder\'on--Zygmund operators, and the corresponding product Hardy and BMO spaces and product T1 theorem were studied in \cite{HLL}.

To state our main result, we write $W^{\sigma, 2}(\R )$ for the usual fractional Sobolev space on $\R$, and $W^{(\sigma_1, \sigma_2),2}(\R \times \R )$ for the product Sobolev space whose norm is given by
\[
\|f\|_{W^{(\sigma_1, \sigma_2),2}(\R \times \R )} = \|(1 +|\xi_1|^{2})^{\sigma_1/2} (1 +|\xi_2|^{2})^{\sigma_2 /2}\widehat{f}(\xi_1, \xi_2)\|_{L^{2}(\R \times \R )},
\]
and $\eta$ for a  $C_0^\infty(\R)$-function supported in $[1/2, 2]$ such that $\eta =1$ in a neighborhood of $1$.

\begin{theorem} \label{main}
Let $Q_1$ and $Q_2$ be the upper dimensions of the spaces $M^{(1)}$ and $M^{(2)}$.
Suppose that $\sigma_{1} > (Q_1 +3)/{2}$ and $\sigma_{2} > (Q_2 +3)/{2}$, and that $ m:[0,\infty)
\times [0,\infty) \to \C $ satisfies
\begin{align} \label{con1}
&\sup_{t_1, t_2 >0} \|\eta (\lambda_1) \eta (\lambda_2)
m(t_1 \lambda_1, t_2 \lambda_2)\|_{W^{(\sigma_1, \sigma_2), 2}(\R \times \R )}<\infty,
\end{align}
\begin{align} \label{con2}
\sup_{t_1>0} \|\eta (\lambda_1) m(t_1 \lambda_1, 0)\|_{W^{\sigma_1, 2}(\R )}<\infty
\end{align}
and
\begin{align} \label{con3}
\sup_{t_2>0} \|\eta (\lambda_2) m(0, t_2 \lambda_2)\|_{W^{\sigma_2,2}(\R )}<\infty.
\end{align}
Then the Marcinkiewicz multiplier $m(\Box_b^{(1)}, \Box_b^{(2)})$, defined on $L^2(\widetilde{M})$ by spectral theory, is a product Calder\'on--Zygmund operator on $\widetilde{M}$ as defined in \cite{HLL}, and hence is bounded on $L^p(\widetilde{M})$ whenever $1 < p < \infty$.
\end{theorem}

We recall that the upper bound \eqref{e7.2} for the heat kernel of the Kohn Laplacian $\Box_b^{(i)}$  is singular on the diagonal, that is, $\{(x,y)\in M^{(i)}\times M^{(i)}: x=y\}$. Thus, our setting is quite different from the setting where the semigroup has Gaussian upper bounds.

\begin{remark}
Suppose that  $m: [0,\infty) \times [0, \infty) \to \C $ satisfies
\begin{equation*}
 \left|\partial_{\lambda_1}^{k_1}\partial_{\lambda_2}^{k_2}
m(\lambda_1, \lambda_2) \right| \leq C_{k_1, k_2} \lambda_1^{-k_1} \lambda_2^{-k_2}
\end{equation*}
for all $k_1, k_2 \in \N_0$ with $k_1 \leq \lfloor \sigma_1 \rfloor +1$ and
$k_2 \leq \lfloor \sigma_2\rfloor +1$, and all $\lambda_1, \lambda_2 >0$.
Then $m$ satisfies conditions \eqref{con1} to \eqref{con3}.
\end{remark}

\begin{remark}
In view of \cite{Str1}, it should be possible to use product Littlewood--Paley theory and domination by the strong maximal function to obtain the $L^p$ boundedness of the multiplier $m(\Box_b^{(1)}, \Box_b^{(2)})$ under the conditions $\sigma_{1} > {(Q_1 +1)}/{2}$ and $\sigma_{2} > {(Q_2 +1)}/{2}$.
However, it is an open question whether one can get rid of the ${1/2}$, since tools such as the Plancherel theorem or restriction estimates are not available.
\end{remark}

The paper is organised as follows. In Section 2 we recall necessary notions on
polynomial domains in $\C^2$ and the Shilov boundary in $\C^{2n}$, as well as the product singular integrals of Journ\'e type.
In Section 3, we establish pointwise kernel estimate for the multipliers when the function $m$ is smooth enough and has compact support.
In the last section we prove the main result.

Throughout this paper, $\N$ denotes the set of positive integers, while $\N_0$ denotes the set of nonnegative integers, that is, $\N_0 = \N \cup \{0\}$.
We will use the symbols $c$ and $C$ to denote positive constants, whose values may vary at every occurrence.
If $f \leq C g$ for some constant $C$ which is independent of the main variables involved, then we write $f \lesssim g$.
If $f \lesssim g \lesssim f$, we write $f \sim g$.

\section{Preliminaries}

In this section we recall some basic notions about polynomial domains in $\C^2$ and Shilov boundaries in $\C ^{2n}$, as well as product Calder\'on--Zygmund operators on $M^{(1)}\times M^{(2)}$.

We start by discussing the single factor case.
Let $\Omega:=\{(z,w)\in\C ^2: \Im(w)>P(z)\}$, where $P$ is a
real, subharmonic, nonharmonic polynomial of finite degree $\nu$, as in \cite{NS}.
Then $M := \partial\Omega$ can be identified with $\C \times
\R =\{(z,t):\ z\in\C , t\in\R \}$.
The basic $(0,1)$ Levi vector field is then $\bar{Z}= { \partial /
{\partial \bar{z} } }- i  { ({\partial P} / {\partial \bar{z}})  }
{ \partial / {\partial t} }  $, and we write $\bar{Z}=X_1+i X_2$.
The real vector fields
$\{X_1, X_2\}$ and their commutators of order at most $m$ span the tangent space to $M$ at each point on $M$.
A curve $\gamma: [0,1] \to M$ is said to be \emph{horizontal} if $\gamma'(t) \in \mathrm{span} \{X_1 , X_2 \}$ at each point, and we define the length of such a horizontal curve by
\[
\mathrm{Len}(\gamma) = \int_0^1 |\gamma'(t)| \,dt,
\]
where $|\gamma'(t)|^2 = |c_1(t)|^2 + |c_2 (t)|^2$ when $\gamma'(t) = c_1 (t) X_1 + c_2 (t) X_2$.
We denote by $\rho$ the control (or Carnot--Carath\'eodory) distance in $M$ generated by $X_1$ and $X_2$, defined by
\begin{align*}
\rho(x, y) := \inf \bigg\{ \mathrm{Len}(\gamma) : \gamma \text{ horizontal}, \gamma(0) = x, \gamma(1) = y \biggr\}.
\end{align*}
We put $B(x,r): = \{y \in M: \rho(x,y) < r\}$.
Let $\mu$ be the Lebesgue measure on $M$ and define $V(x,r) :=\mu(B(x,r))$.
Then
\begin{equation}\label{measure}
V(x,\delta) = \biggl( \sum_{k=2}^m \Lambda_k(x)\delta^k \biggr)\delta^2,
\end{equation}
where the $\Lambda_k(x)$ are continuous, non-negative functions on $M$ (see \cite[Section 4.1]{NS01}).
Moreover, if $Q = m + 2$, then for all $x\in M, \lambda\geq 1$ and $\delta >0$,
\begin{align}
\label{ccc}
 V(x, \lambda \delta) &\leq C\lambda^{Q} V(x,   \delta) \\
\noalign{\noindent{and}}
\label{reverse doubling}
V(x,\lambda \delta) &\geq C \lambda^4 V(x,\delta)
\end{align}
for all $x\in M$, $\lambda\geq 1$ and $\delta >0$,

Denote by $X_j^{\ast}$ the formal adjoint of $X_j$, that is,
\[
\langle X_j^{\ast} \varphi, \psi\rangle= \langle \varphi, X_j\psi\rangle,
\]
where $\langle \varphi,  \psi\rangle=\int_M \varphi(x)  \psi(x) \,dx$ for all $\varphi, \psi\in C^{\infty}(M)$.
In general, $X_j^{\ast}=-X_j +a_j$, where $a_j\in C^{\infty}(M)$.
The Kohn  Laplacian $\Box_b$ on $M$ is formally defined by
\begin{align*}
\Box_b:= X_1^{\ast}X_1 + X_2^{\ast}X_2.
\end{align*}
Set the heat operator $T_t:= e^{-t\Box_b}$ for $t>0$.
As recalled in \eqref{e7.2},
the kernel of $T_t$ satisfies the non-Gaussian upper bounds  in terms of the control distance $\rho$ on $M$.

Since $\Box_b$ is a self-adjoint operator, it admits a spectral decomposition $E(\lambda)$;
in particular, $E(0) = \pi$, where $\pi$ is the Szeg\H{o} projection on $M$.
Hence, for any bounded Borel measurable function $m : [0,\infty) \to \C $,
we may define
\begin{align*}
m (\Box_b )= \int_{[0,\infty)}m (\lambda) \,dE (\lambda)
\end{align*}
and, with an abuse of notation, we define
\begin{align*}
m(\hatbox_b) =\int_{(0,\infty)} m (\lambda) \,dE (\lambda),
\end{align*}
so that
\begin{align*}
m(\hatbox_b) = (1-\pi) m (\Box_b) = m (\Box_b)- m (0) \pi.
\end{align*}

To describe higher-order derivatives on $M$ we need more notation.
For each $j \in \N$, we set
\begin{align*}
\mathcal{I}_j:= \{1,2\}^{j} = \{(\alpha_1, \cdots, \alpha_j):  \alpha_{1}, \cdots, \alpha_j  \in \{1,2\}\}.
\end{align*}
For $\alpha =(\alpha_1, \cdots, \alpha_j)\in \mathcal{I}_j$, we define $|\alpha|= j$ and call $|\alpha|$ the length of the multi-index $\alpha$; we also set
\[
X_{\alpha}:= X_{\alpha_1} \cdots X_{\alpha_j}.
\]
For instance, if $\alpha= (1,2,2,1)$, then $|\alpha| =4$ and $X_{\alpha} = X_1 X_2 X_2 X_1$.
By convention, let $\mathcal{I}_0=\{(0)\}$, and for $\alpha \in \mathcal{I}_0$ (that is, $\alpha = (0)$),
we define $|\alpha| = 0$ and $X_\alpha = I$ (the identity operator).  We also put
\begin{align*}
\mathcal{I} = \bigcup_{j =0}^\infty \mathcal{I}_j,
\end{align*}
and write $f \in L^p_k(M)$ if $X_\alpha f \in L^p(M)$ whenever $|\alpha| \leq k$.

Anisotropic smoothing operators (NIS operators for short) were introduced in \cite{NRSW}.
Here we use the definition from \cite[Definition 5.10]{K}.

\begin{definition} \label{NIS}
Let $T: C_0^{\infty}(M) \to C^\infty (M)$ be a linear operator, with Schwartz kernel
$K_T (x, y)$. We say that $T$ is an NIS operator, smoothing of order $r$, if $K_T$ is $C^\infty$ away from the
diagonal of $M \times M$ and the following conditions are satisfied:
\begin{enumerate}
\item
For all $s \geq 0$, there exist parameters $a(s) <\infty$ and $b<\infty$ such that if $\zeta, \zeta' \in C_0^\infty (M)$ and $\zeta' \equiv 1$ on the support of $\zeta$, then there exists $C = C(s, \zeta, \zeta')$ such that
\begin{align*}
\|\zeta Tf \|_{L^2_s(M)} \leq C\left(\|\zeta' f\|_{L^2_{a(s)}(M)} +\|f\|_{L_b^2(M)}\right)
\qquad\forall f \in C_0^\infty (M).
\end{align*}

\item
There exist constants $C_{\alpha, \beta}$ such that for $x \neq y$,
\begin{align*}
\left|(X_\alpha)_x (X_\beta)_y K_T (x,y) \right| \leq C_{\alpha, \beta} \frac{\rho(x,y)^{r -|\alpha| -|\beta|}}{V(x,\rho(x,y))}.
\end{align*}

\item
For each integer $\ell \geq 0$,
there is an integer $N= N(\ell) \geq 0$ and a constant $C = C(\ell)$ such
that if $\phi \in C_0^\infty (B(x, \delta))$, then\footnote{What does $D^\beta \phi$ mean?}
\begin{align*}
\sum_{|\alpha| = \ell} \left| X_\alpha T(\phi) (x) \right| \leq C\delta^{r -\ell} \sup_{y \in M}
\sum_{|\beta| \leq N} \delta^{|\beta|} |X_\beta \phi (y)|.
\end{align*}

\item
The above conditions also hold for the adjoint operator $T^\ast$.
\end{enumerate}
\end{definition}

It is shown in \cite{NRSW, K} that $\pi$ is an NIS operator of order zero and all NIS operators
of order zero are bounded on $L^p(M)$.

Nagel and Stein \cite{NS1} introduced a class of singular integrals on $M$ which includes NIS operators of order zero and
proved that they are still bounded on $L^p(M)$ for $1<p <\infty$. The main difference between their singular integral operators and NIS operators of order zero is that the former are not assumed to satisfy the condition (i) of Definition \ref{NIS}.
In the same paper, products of such singular integral operators are also studied.

For NIS operators and the singular integrals operators of Nagel and Stein  \cite{NS1}, condition (iii) is called a cancellation condition; it guarantees boundedness on $L^2(M)$.
For spectral multipliers, the $L^2$-boundedness holds from the definition via spectral theory.

We will need the definition of single-parameter Calder\'{o}n--Zygmund operators on $M$.
A continuous function $K(x,y)$ defined on  $M\times M \backslash \{ (x, y): x=y\}$  is called a \emph{Calder\'on--Zygmund kernel} if there exists a constant $C>0$ and a regularity exponent $\varepsilon\in (0,1]$ such that
\begin{align}
\tag{i} |K(x,y)|&\leq C V(x,y)^{-1};\\
\tag{ii} \vert K(x,y)-K(x,y')\vert &\leq C
\biggl(\frac{\rho(y,y')}{\rho(x,y)}\biggr)^{\varepsilon} V(x,y)^{-1}
             \quad\text{if}\quad \rho(y,y')\le \frac{\rho(x,y)}{2};\\
\tag{iii} \vert K(x,y)-K(x',y)\vert &\leq C
\biggl(\frac{\rho(x,x')}{\rho(x,y)}\biggr)^{\varepsilon} V(x,y)^{-1}
             \quad\text{if}\quad \rho(x,x')\le \frac{\rho(x,y)}{2}.
\end{align}
The smallest such constant $C$ is denoted by $|K|_{CZ}$. We say
that an operator $T$ is a \emph{singular integral operator} associated with a Calder\'on--Zygmund kernel
$K$ if the operator $T$ is a continuous linear operator from
$C^\infty_0(M)$   into its dual such that
\[
\langle Tf, g \rangle=\iint g(x) K(x,y) f(y) \,d\mu(y) \,d\mu(x)
\]
for all functions $f, g\in C^\infty_0(M)$ with disjoint supports;
$T$ is said to be a \emph{Calder\'on--Zygmund operator} if it
extends to be a bounded operator on $L^2(M)$. If $T$ is a
Calder\'on--Zygmund operator associated with a kernel $K$, its
operator norm is defined by
\[
\|T\|_{CZ}=\|T\|_{L^2\to
L^2}+ | K|_{CZ}.
\]

We now turn to the product case, that is, the case of two factors.
Let $M^{(k)}$ be the boundary of an unbounded model polynomial domain $\Omega^{(k)}$, where $k=1,2$.
We will use superscripts to indicate operators or quantities on $M^{(k)}$.
For example, the Kohn Laplacian, control distance, measure, volume of a ball and higher-order derivatives will be denoted by $\Box_b^{(k)}, \rho^{(k)}(\cdot, \cdot)$, $d\mu^{(k)}$, $V^{(k)}(x,r)$ and $X_\alpha^{(k)}$ respectively.

Now we recall the definition of \emph{product singular integral operators}
on $\widetilde{M}  = M^{(1)} \times M^{(2)}$.
A linear operator $T: C^\infty_0({\widetilde M})\mapsto [C^\infty_0({\widetilde M})]'$ is said to be a singular integral operator if there exist Calder\'on--Zygmund-operator-valued operators $K_1$ on $M^{(2)}$ and $K_2$ on $M^{(1)}$ with the following properties:
\begin{align*}
\tag{i}
&\langle g\otimes k, Tf\otimes h\rangle
=\iint g(x_1)\langle k, K_1(x_1,y_1)h\rangle f(y_1)
                \,d\mu^{(1)}(x_1) \,d\mu^{(1)}(y_1) \\
\noalign{\noindent{for all $f, g\in C^\infty_0( M^{(1)})$ and $h, k\in C^\infty_0(M^{(2)})$ with $\supp f \cap \supp g =\varnothing$;}}
\tag{ii}
&\langle k\otimes g, Th\otimes f\rangle
=\iint g(x_2)\langle k, K_2(x_2,y_2)h\rangle f(y_2) \,d\mu^{(2)}(x_2) \,d\mu^{(2)}(y_2)  \\
\noalign{\noindent{for all $f, g\in C^\infty_0(M^{(2)})$ and $h, k\in C^\infty_0(M^{(1)})$, with $\supp f \cap \supp g=\varnothing$;}}
\tag{iii}
&\hspace{2.3cm} \| K_k (x_k,y_k) \|_{CZ}
\leq C V^{(k)}(x_k,y_k)^{-1};\\
\tag{iv}
&\|K_k (x_k,y_k)-K_k(x_k,y'_k)\|_{CZ}
\leq C
\biggl(\frac{\rho^{(k)}(y_k,y'_k)}{\rho^{(k)}(x_k,y_k)}\biggr)^{\varepsilon_k}
        V^{(k)}(x_k,y_k)^{-1} \\
\noalign{\noindent{if $\rho^{(k)}(y_k,y'_k)\le \dfrac{\rho^{(k)}(x_k,y_k)}{2}$;}}
\tag{v}
&\|K_k(x_k,y_k)-K_k(y'_k,y_k)\|_{CZ}
\leq C \biggl(\frac{\rho^{(k)}(x_k,y'_k)}{\rho^{(k)}(x_k,y_k)}\biggr)^{\varepsilon_k}
        V^{(k)}(x_k,y_k)^{-1} \\
\noalign{\noindent{if $\rho^{(k)}(x_k,y'_k)\le \dfrac{\rho^{(k)}(x_k,y_k)}{2}\,.$}}
\end{align*}

A product singular integral operator $T$ is said to be a \emph{product Calder\'on--Zygmund operator} if $T$ is bounded on $L^{2}(\widetilde{M})$.

\begin{proposition}[\cite{HLL}] \label{mmm0}
Suppose that $T$ is a product Calder\'on--Zygmund operator.
Then $T$ extends to a bounded operator on $L^{p}(\widetilde{M})$ when $1 <p <\infty$.
\end{proposition}

For convenience in our application, we also provide the following version, where the difference is replaced by derivative in the regularity condition of the kernel.
This is a special case of the above proposition.

\begin{proposition} \label{mmm}
Suppose that $T$ is bounded on $L^{2}(\widetilde{M})$, where $\widetilde{M} = M^{(1)} \times M^{(2)}$.
Suppose further that the distribution kernel $K(x_1, x_2,y_1, y_2)$ of $T$ is $C^{\infty}$ away from the ``cross''
$\{(x_1, x_2, y_1, y_2): (x_1 - y_1)(x_2 -y_2) = 0\}$
and satisfies the following additional properties:\\[2pt]
(i)\enspace $\langle T(\varphi_1 \otimes \varphi_2), \psi_1 \otimes \psi_2 \rangle$ is given by
\begin{align*}
  & \iiiint  K (x_1,x_2,y_1,y_2) \varphi_1 (y_1)
\varphi_2 (y_2) \psi_1 (x_1) \psi_2 (x_2) \\
  &\hspace{4.6cm}
        \,d\mu^{(1)}(x_1) \,d\mu^{(2)}(x_2) \,d\mu^{(1)}(y_1) \,d\mu^{(2)}(y_2)
\end{align*}
whenever $\varphi_1, \psi_1 \in C_0^{\infty} (M^{(1)})$ have disjoint supports and $\varphi_2, \psi_2 \in C_0^{\infty} (M^{(2)})$ have disjoint supports.
\\[2pt]
(ii)\enspace
For all functions $f_2,g_2\in L^2(M^{(2)})$, there exists a constant $C>0$ independent of $f_2,g_2$ such that for all $\alpha,\beta \in \{0,1\}$,
\begin{align*}
&\left|\int_{M^{(2)}}\int_{M^{(2)}}(X_{\alpha}^{(1)})_{x_1}(X_{\beta}^{(1)})_{y_1}K(x_1, x_2,y_1, y_2)f_2(x_2)g_2(y_2)\,d\mu^{(2)}(x_2)\,d\mu^{(2)}(y_2) \right|\\
&\hspace{5.1cm} \leq \frac{C\rho^{(1)}(x_1,y_1)^{-|\alpha|-|\beta|}}{V^{(1)}(x_1, y_1)} \|f_2\|_2\|g_2\|_2.
\end{align*}
\\[2pt]
(iii)\enspace
The kernel $K(x_1, x_2, y_1, y_2)$ satisfies the estimate
\begin{align*}
&|(X_{\alpha}^{(1)})_{x_1} (X_{\beta}^{(1)})_{y_1} (X_{\alpha'}^{(2)})_{x_2}
 (Y^{(2)}_{\beta'})_{y_2} K(x_1, x_2, y_1, y_2)| \\
 &\hspace{4cm} \lesssim \frac{\rho^{(1)} (x_1, y_1)^{-|\alpha| - |\beta|} \rho^{(2)}(x_2, y_2)^{-|\alpha'|-|\beta'|}}{V^{(1)}(x_1, y_1)V^{(2)}(x_2,y_2)} .
\end{align*}
\\[2pt]
(iv)\enspace
Conditions (2) and (3) hold when the indices $1$ and $2$ are interchanged, that is, if the roles of $M^{(1)}$ and $M^{(2)}$ are interchanged.
\\[2pt]
(v)\enspace
Conditions (2) to (4) hold for the three transposes of $T$, that is, the operators that arise by interchanging $x_1$ and $y_1$, or interchanging $x_2$ and $y_2$, or interchanging both pairs of variables.
\\[2pt]
Then $T$ extends to a bounded operator on $L^{p}(\widetilde{M})$ when $1 <p <\infty$.
\end{proposition}

\begin{remark}
If $T_1$ and $T_2$ are (single-factor) singular integral operators on $M^{(1)}$ and  $M^{(2)}$, then $T = T_1 \otimes T_2$ satisfies the above assumptions.
Here $T^{\varphi_2,x_2}$ is equal to $T_1$ multiplied by the factor $T_2(\varphi_2)(x_2)$.
\end{remark}

\section{Functional calculus}
In this section, our aim is to prove two pointwise kernel estimates for the multipliers where the multiplier function is smooth enough and with compact support.

We begin with two lemmas which state some useful elementary estimates for functions.
For one-parameter version of these estimates, see Street \cite{Str1, Str2}.
We write $\MF$ to indicate the Fourier transformation: subscripts $1$ and $2$ indicate Fourier transforms with respect to the first or second variables only.

\begin{lemma} \label{elem}
Suppose $m: [0,\infty) \times [0,\infty) \to \C $ is a function with $\supp m \subset [1/4, 4] \times [1/4, 4]$, and let $\phi$ be a smooth function on $\R $ such that
\begin{align*}
\phi(\xi) =
\begin{cases}
0& \text{if  $|\xi| \leq 1/4$} \\
1& \text{if  $|\xi| \geq 1/2$}.
\end{cases}
\end{align*}
Set $\psi(\lambda_1, \lambda_2) := m(\lambda_1^2,\lambda_2^2)$.
\\[2pt]
(i)\enspace
Define $F_{r_1,r_2}(\lambda_1, \lambda_2):=  \psi (r_1\lambda_1, r_2 \lambda_2)$.
Then, for any $\varepsilon>0$ and any $N_1, N_2 \in \N$,
there exists a constant $C_{\varepsilon, N_1,N_2}$ such that
\begin{align*}
|F_{r_1,r_2} (\lambda_1, \lambda_2)|  \leq
C_{\varepsilon, N_1,N_2}\|m\|_{W^{(\varepsilon+1/2, \varepsilon+1/2),2}(\R \times \R )}(1 +
r_1\lambda_1 )^{-N_1}(1 + r_2\lambda_2 )^{-N_2}
\end{align*}
for all $r_1, r_2, \lambda_1,\lambda_2 >0$.
\\[2pt]
(ii)\enspace
Define $F_{s_1,r_1,r_2}(\lambda_1, \lambda_2)$ by the formula
\begin{align*}
\MF_1F_{s_1,r_1,r_2} (\xi_1, \lambda_2)
= \phi\left(\frac{\xi_1}{s_1}\right) \frac{1}{r_1}\MF_1{\psi}\left(\frac{\xi_1}{r_1},r_2 \lambda_2\right).
\end{align*}
Then, for any $\varepsilon>0$, any $N_1, N_2 \in \N$ and any $\sigma_1 \in \R $,
there exists a constant $C_{\varepsilon, \sigma_1, N_1,N_2}$ such that
\begin{align*}
|F_{s_1,r_1,r_2} (\lambda_1, \lambda_2)|  &\leq
C_{\varepsilon, \sigma_1, N_1,N_2}\|m\|_{W^{(\sigma_1, \varepsilon+1/2),2}(\R \times \R )} \\
&\quad \times \left[(1+s_1\lambda_1)^{-N_1}\left( 1+ \frac{s_1}{r_1} \right)^{(N_1 + \varepsilon + {1/2})-\sigma_1}\right]
(1 +r_2 \lambda_2 )^{-N_2}
\end{align*}
for all $r_1, r_2,s_1, \lambda_1,\lambda_2 >0$.
\\[2pt]
(iii)\enspace
Define $F_{s_2,r_1,r_2}(\lambda_1, \lambda_2)$ by the formula
\begin{align*}
\MF_2{ F_{s_2,r_1,r_2}} (\lambda_1, \xi_2)= \phi\left(\frac{\xi_2}{s_2}\right) \frac{1}{r_2}\MF_2{\psi}\left(r_1 \lambda_1, \frac{\xi_2}{r_2}\right).
\end{align*}
Then, for any $\varepsilon>0$, any $N_1, N_2 \in \N$, and any $\sigma_2 \in \R $,
there exists a constant $C_{\varepsilon, \sigma_2, N_1,N_2}$ such that
\begin{align*}
|F_{s_2,r_1,r_2} (\lambda_1, \lambda_2)|
&\leq
C_{\varepsilon, \sigma_2, N_1,N_2}\|m\|_{W^{(\varepsilon+1/2, \sigma_2),2}(\R \times \R )}\\
&\qquad \times (1 + r_1 \lambda_1 )^{-N_1}\left[(1+s_2 \lambda_2)^{-N_2}\left( 1+ \frac{s_2}{r_2} \right)^{(N_2 + \varepsilon + {1/2})-\sigma_2}
\right]
\end{align*}
for all $r_1, r_2,s_2, \lambda_1,\lambda_2 >0$.
\\[2pt]
(iv)\enspace
Define $F_{s_1,s_1,r_1,r_2}(\lambda_1, \lambda_2)$ by the formula
\begin{align*}
\MF{F_{s_1,s_2,r_1,r_2}} (\xi_1, \xi_2)= \phi\left(\frac{\xi_1}{s_1}\right) \phi\left(\frac{\xi_2}{s_2}\right)\frac{1}{r_1 r_2}
\MF{\psi}\left(\frac{\xi_1}{r_1}, \frac{\xi_2}{r_2}\right).
\end{align*}
Then, for any $\varepsilon>0$, any $N_1, N_2 \in \N$, and any $\sigma_1, \sigma_2 \in \R $,
there exists a constant $C_{\varepsilon, \sigma_1, \sigma_2, N_1,N_2}$ such that
\begin{align*}
|F_{\sigma_1, \sigma_2, s_1,s_2,r_1,r_2} (\lambda_1, \lambda_2)|  &\leq C_{\varepsilon,\sigma_1, \sigma_2, N_1,N_2}\|m\|_{W^{(\sigma_1, \sigma_2),2}(\R \times \R )}\\
 & \quad\times \left[(1+s_1 \lambda_1)^{-N_1}\left( 1+ \frac{s_1}{r_1} \right)^{(N_1 + \varepsilon + {1/2})-\sigma_1}\right]\\
&\quad \times \left[ (1+ s_2\lambda_2)^{-N_2}
\left( 1+ \frac{s_2}{r_2} \right)^{(N_2 + \varepsilon + {1/2})-\sigma_2}\right]
\end{align*}
for all $r_1, r_2,s_1, s_2,\lambda_1,\lambda_2 >0$.
\end{lemma}
\begin{proof}
We first prove (i). Using the fact that $W^{(\varepsilon+1/2, \varepsilon+1/2),2}(\R  \times \R )
\subset L^{\infty}(\R  \times \R ) \cap C(\R  \times \R )$  and
$\supp \psi \subset [1/2,2]\times [1/2,2]$, we see that
\begin{align*}
|F_{r_1,r_2}(\lambda_1, \lambda_2)| &\leq 3^{N_1} 3^{N_2}\|\psi
\|_{L^{\infty}(\R  \times \R )} (1 + |\lambda_1|)^{-N_1} (1+|\lambda_2|)^{-N_2}\\
& \leq C_{\varepsilon, N_1, N_2} \|\psi\|_{ W^{(\varepsilon+1/2, \varepsilon+1/2),2}(\R  \times \R )} (1 + |\lambda_1|)^{-N_1} (1+|\lambda_2|)^{-N_2}\\
& \leq C_{\varepsilon, N_1, N_2} \|m\|_{ W^{(\varepsilon+1/2, \varepsilon+1/2),2}(\R  \times \R )} (1 + |\lambda_1|)^{-N_1} (1+|\lambda_2|)^{-N_2}.
\end{align*}

Next we prove (ii).
Fix a function $\eta \in C_{0}^{\infty}(\R )$ such that $\eta(\lambda_2) =1$ for all $\lambda_2 \in [1/2, 2]$.
We consider four cases.

\subsubsection*{Case 1.} $s_1\lambda_1 \leq 1$ and $s_1 \leq r_1$.
By the Cauchy--Schwarz inequality,
\begin{align*}
\left|F_{s_1,r_1,r_2} (\lambda_1, \lambda_2) \right|
&= \left|\int_{\R } \phi\left(\frac{\xi_1}{s_1}\right)\frac{1}{r_1}\MF_1{\psi}
\left(\frac{\xi_1}{r_1}, r_2 \lambda_2\right)e^{i \lambda_1 \xi_1}d\xi_1 \right|\\
&= \frac{s_1}{r_1}\left|\int_{|\xi_1| \geq {1/4}} \phi(\xi_1)\MF_1{\psi}\left( \frac{s_1\xi_1}{r_1},
r_2 \lambda_2\right)
e^{i s_1 \lambda_1 \xi_1}d\xi_1\right|\\
&\lesssim \frac{s_1}{r_1} \int_{|\xi_1| \geq {1/4}}
\left|\MF_1{\psi}\left(\frac{s_1\xi_1}{r_1}, r_2\lambda_2\right)\right| d\xi_1  .
\end{align*}
Now $\supp \psi \subset [1/2,2] \times [1/2,2]$, so $\supp \MF_1{\psi}\bigl(\frac{s_1\xi_1}{r_1}, \cdot \bigr) \subseteq [1/2,2]$
for all $s_1, r_1 >0$ and $\xi_1 \in \R $.
Hence $\MF_1{\psi}\bigl(\frac{s_1\xi_1}{r_1}, \lambda_2 r_2 \bigr)
= \eta(\lambda_2 r_2) \MF_1{\psi}\bigl(\frac{s_1\xi_1}{r_1}, \lambda_2r_2 \bigr)$, and it follows that
\begin{align*}
\left|F_{s_1,r_1,r_2} (\lambda_1, \lambda_2) \right|& \lesssim \frac{s_1}{r_1} |\eta(r_2 \lambda_2)|
 \int_{|\xi_1| \geq {1/4}}
\left|\MF_1{\psi}\left(\frac{s_1\xi_1}{r_1},r_2^2\lambda_2\right)
\right| d\xi_1 \\
&\leq \frac{s_1}{r_1}|\eta(r_2 \lambda_2)|\int_{|\xi_1| \geq {1/4}}
\left(1+ \left| \frac{s_1\xi_1}{r_1}\right| \right)^{\sigma_1 -
(\varepsilon + {1/2})}
\left|\MF_1{\psi}\left(\frac{s_1 \xi_1}{r_1},r_2^2\lambda_2\right)
\right| d\xi_1 \\
&\leq |\eta(r_2 \lambda_2)|\left(\int_{|\xi_1| \geq {1/4}}
\left|\left(1+ \left| \frac{s_1\xi_1}{r_1}\right| \right)^{\sigma_1}
\MF_1{\psi}\left(\frac{s_1\xi_1}{r_1},r_2^2\lambda_2\right)
\right|^{2}  d \left( \frac{s_1\xi_1}{r_1}\right)\right)^{1/2} \\
&\hspace{4cm}\times \left( \int_{|\xi_1| \geq {1/4}} \left|1+  \frac{s\xi_1}{r_1}\right|^{ -(1 + 2\varepsilon)}
 d \left( \frac{s_1\xi_1}{r_1}\right)\right)^{1/2} \\
& \lesssim |\eta(r_2\lambda_2)|\left(\int_{\R } \left|(1+ |\xi |)^{\sigma_1}
\MF_1{\psi}\left(\xi_1,r_2^2\lambda_2\right)
\right|^{2}  d \xi\right)^{1/2} .
\end{align*}
Using the continuous inclusion $W^{\varepsilon+1/2,2}(\R ) \subset L^{\infty}(\R )$, we see that
\begin{align*}
\left|\MF_1{\psi}\left( \xi_1,r_2^2\lambda_2
\right)\right|^2 \lesssim \|\MF_1{\psi}(\xi_1, \cdot)\|_{L^{\infty}(\R )}^2
 \lesssim \int_{\R }\big|(1 + |\xi_2|)^{\varepsilon + {1/2}}  \MF{\psi}(\xi_1, \xi_2)\big|^2 d\xi_2.
\end{align*}
Hence by Fubini's theorem,
\begin{align*}
&|F_{s_1,r_1,r_2} (\lambda_1, \lambda_2) | \\
&\qquad\lesssim  |\eta(r_2 \lambda_2)| \left[\int_{\R \times \R }
\left|\left(1+\left|\xi_1\right|\right)^{\sigma_1 }(1+|\xi_2|)^{
\varepsilon + {1/2}}\MF{\psi}(\xi_1, \xi_2)\right|^{2}  d\xi_1d\xi_2\right]^{1/2}\\
&\qquad\lesssim  (1 +r_2^2 \lambda_2)^{-N_2} \|\psi\|_{W^{(\sigma_1, \varepsilon+1/2),2}(\R  \times \R )}\\
&\qquad\lesssim \|m\|_{W^{(\sigma_1, \varepsilon+1/2),2}(\R \times \R )}
\left[(1+s_1 \lambda_1)^{-N_1}\left( 1+ \frac{s_1}{r_1} \right)^{(N_1 + \varepsilon + {1/2})-\sigma_1}\right]
(1 + r_2 \lambda_2)^{-N_2},
\end{align*}
 where for the last inequality we used the estimate $ s_1 \lambda_1 \sim 1 + s_1 \lambda_1$ and $1+ {s_1}/{r_1} \sim 1$.

\subsubsection*{Case 2.} $s_1\lambda_1 \geq 1$ and ${s_1} \leq {r_1}$.
In this case, by elementary calculus,
\begin{align*}
\left|F_{s_1,r_1,r_2} (\lambda_1, \lambda_2) \right|
&= \left|\int_{\R } \phi\left(\frac{\xi_1}{s_1}\right)\frac{1}{r_1}\MF_1{\psi}
\left(\frac{\xi_1}{r_1},r_2^2 \lambda_2\right)e^{i \lambda_1 \xi_1}d\xi_1 \right|\\
&= \frac{s_1}{r_1}\left|\int_{|\xi_1| \geq {1/4}} \phi(\xi_1)\MF_1{\psi}\left( \frac{s_1\xi_1}{r_1},r_2\lambda_2\right)
e^{i s_1 \lambda_1 \xi_1}d\xi_1\right|\\
&= (s_1 \lambda_1)^{-N_1} \frac{s_1}{r_1}\left|  \int_{|\xi_1| \geq {1/4}} \phi(\xi_1)\MF_1{\psi}
\left( \frac{s_1 \xi_1}{r_1},r_2\lambda_2\right)\partial_{\xi_1}^{N_1}[e^{i s_1 \lambda_1 \xi_1}]d\xi_1\right|\\
&= (s_1 \lambda_1)^{-N_1} \frac{s_1}{r_1}  \left|\int_{|\xi_1| \geq {1/4}} \partial_{\xi_1}^{N_1}\left[ \phi(\xi_1)
\MF_1{\psi}\left( \frac{s_1\xi_1}{r_1},r_2\lambda_2\right)\right]
e^{i s_1 \lambda_1 \xi_1}d\xi_1\right|\\
& \leq (s_1 \lambda_1)^{-N_1} \sum_{k_1 + k_2 =N_1} \left( \frac{s_1}{r_1}\right)^{1+ k_2}\\
&\quad \quad \times  \left|\int_{|\xi_1| \geq {1/4}}
\left[\partial_{\xi_1}^{N_1} \phi (\xi_1)\right]\left[(\partial_{\xi_1}^{k_2} \MF_1{\psi})\left(\frac{s_1\xi_1}{r_1},r_2\lambda_2
\right)\right] e^{i s_1\lambda_1 \xi_1}d\xi_1\right|\\
&= (s_1 \lambda_1)^{-N_1} \sum_{k_1 + k_2 =N_1} \left( \frac{s_1}{r_1}\right)^{1+ k_2}|\eta(r_2\lambda_2)|\\
&\hspace{2cm} \times  \left|\int_{|\xi_1| \geq {1/4}}
\left[\partial_{\xi_1}^{k_1} \phi (\xi_1)\right]\MF_1{\varphi_{k_2}}\left(\frac{s_1\xi_1}{r_1},r_2\lambda_2
\right) e^{i s_1\lambda_1 \xi_1}d\xi_1\right|,
\end{align*}
where $\varphi_{k_2}$ is the function on $\R  \times \R $ given by
$\varphi_{k_2}(\lambda_1, \lambda_2)=(-i \lambda_1)^{k_2} \psi(\lambda_1,\lambda_2)$.
Much as in Case~1, $\MF_1{\varphi_{k_2}}({s_1\xi_1}/{r_1},r_2\lambda_2)
= \eta(\lambda_2 r_2 ) \MF_1{\varphi_{k_2}}({s_1\xi_1}/{r_1},r_2\lambda_2)$.
It follows that
\begin{align*}
&\left|F_{s_1,r_1,r_2} (\lambda_1, \lambda_2) \right|\\
&\qquad\leq (s_1 \lambda_1)^{-N_1} \sum_{k_1 + k_2 =N_1} \left( \frac{s_1}{r_1}\right)^{1+ k_2}|\eta(r_2\lambda_2)
| \int_{|\xi_1| \geq {1/4}}
\left|\MF_1{\varphi_{k_2}}\left(\frac{s_1\xi_1}{r_1}, r_2\lambda_2\right)
\right| d\xi_1 \\
&\qquad\leq (s_1 \lambda_1)^{-N_1} \sum_{k_1 + k_2 =N_1} \left( \frac{s_1}{r_1}\right)^{1+ k_2}|\eta(r_2\lambda_2)| \\
&\qquad\hspace{2cm} \times \int_{|\xi_1| \geq {1/4}}\left(1+\left|\frac{s_1\xi_1}{r_1}\right|\right)^{\sigma_1 - (\varepsilon + {1/2})}
\left|\MF_1{\varphi_{k_2}}\left(\frac{s_1\xi_1}{r_1},r_2\lambda_2\right)
\right| d\xi_1 \\
&\qquad\leq  (s_1 \lambda_1)^{-N_1} \sum_{k_1 + k_2 =N_1} \left( \frac{s_1}{r_1}\right)^{k_2 }|\eta(r_2\lambda_2)| \\
&\qquad\hspace{2cm} \times \left[\int_{|\xi_1| \geq {1/4}}
\left|\left(1+\left|\frac{s_1\xi_1}{r_1}\right|\right)^{\sigma_1 }\MF_1{\varphi_{k_2}}\left(\frac{s_1\xi_1}{r_1},r_2\lambda_2
\right)\right|^{2}  d\left(\frac{s_1\xi_1}{r_1}\right)\right]^{1/2}\\
&\qquad\hspace{2cm} \times \left[\int_{\R } \left(1+ \left| \frac{s_1\xi_1}{r_1}\right| \right)^{-1-2\varepsilon} d\left(\frac{s_1\xi_1}{r_1}\right)\right]^{1/2}\\
&\qquad\leq  (s_1 \lambda_1)^{-N_1} \sum_{k_1 + k_2 =N_1} \left( \frac{s_1}{r_1}\right)^{k_2 }|\eta(r_2\lambda_2)| \\
&\qquad\hspace{2cm} \times \left[\int_{\R }
\left|\left(1+\left|\xi_1\right|\right)^{\sigma_1 }\MF_1{\varphi_{k_2}}(\xi_1, r_2\lambda_2)\right|^{2}  d\xi_1\right]^{1/2}\\
&\qquad\lesssim  (s_1 \lambda_1)^{-N_1} (1+r_2\lambda_2)^{-N_2}\|\varphi_{k_2}\|_{W^{(\sigma_1, \varepsilon+1/2),2}(\R  \times \R )}\\
&\qquad\lesssim  (s_1 \lambda_1)^{-N_1} (1+r_2\lambda_2)^{-N_2}\|\psi\|_{W^{(\sigma_1, \varepsilon+1/2),2}(\R  \times \R )}\\
&\qquad\lesssim \|m\|_{W^{(\sigma_1, \varepsilon+1/2),2}(\R \times \R )}
\left[(1+s_1 \lambda_1)^{-N_1}\left( 1+ \frac{s_1}{r_1} \right)^{(N_1 + \varepsilon + {1/2})-\sigma_1}\right] (1 +r_2\lambda_2)^{-N_2}
\end{align*}
where for the last inequality we used the estimate $ 1+ s_1 \lambda_1 \sim 1$ and $1+ {s_1}/{r_1} \sim 1$.

\subsubsection*{Case 3.} $ s_1\lambda_1 \leq 1$ and $s_1 \geq r_1$.
This case is similar to Case 2, and we skip the details.

\subsubsection*{Case 4.} $s_1 \lambda_1 \geq 1$ and ${s_1} \geq {r_1}$.
Taking $\varphi_{k_2}$ as in Case 2, we see that
\begin{align*}
&\left|F_{s_1,r_1,r_2} (\lambda_1, \lambda_2) \right| \\
&\qquad= \left|\int_{\R } \phi\left(\frac{\xi_1}{s_1}\right)\frac{1}{r_1}\MF_1{\psi}
\left(\frac{\xi_1}{r_1},r_2\lambda_2\right)e^{i \lambda_1 \xi_1}d\xi_1 \right|\\
&\qquad= \frac{s_1}{r_1}\left|\int_{|\xi_1| \geq {1/4}} \phi(\xi_1)\MF_1{\psi}\left( \frac{s_1\xi_1}{r_1},r_2\lambda_2\right)
e^{i s_1 \lambda_1 \xi_1}d\xi_1\right|\\
&\qquad= (s_1 \lambda_1)^{-N_1} \frac{s_1}{r_1}\left|  \int_{|\xi_1| \geq {1/4}} \phi(\xi_1)\MF_1{\psi}
\left( \frac{s_1  \xi_1}{r_1},r_2\lambda_2\right)\partial_{\xi_1}^{N_1}[e^{i s_1 \lambda_1 \xi_1}]d\xi_1\right|\\
&\qquad= (s_1 \lambda_1)^{-N_1} \frac{s_1}{r_1}  \left|\int_{|\xi_1| \geq {1/4}} \partial_{\xi_1}^{N_1}\left[ \phi(\xi_1)
\MF_1{\psi}\left( \frac{s_1\xi_1}{r_1},r_2\lambda_2\right)\right]
e^{i s_1 \lambda_1 \xi_1}d\xi_1\right|\\
&\qquad\leq (s_1 \lambda_1)^{-N_1} \sum_{k_1 + k_2 =N_1} \left( \frac{s_1}{r_1}\right)^{1+ k_2}\\
&\qquad\qquad \times  \left|\int_{|\xi_1| \geq {1/4}}
\left[\partial_{\xi_1}^{N_1} \phi (\xi_1)\right]\left[(\partial_{\xi_1}^{k_2} \MF_1{\psi})\left(\frac{s_1\xi_1}{r_1},r_2\lambda_2
\right)\right] e^{i s_1\lambda_1 \xi_1}d\xi_1\right|\\
&\qquad= (s_1 \lambda_1)^{-N_1} \sum_{k_1 + k_2 =N_1} \left( \frac{s_1}{r_1}\right)^{1+ k_2}|\eta(r_2\lambda_2)|\\
&\qquad\qquad \times  \left|\int_{|\xi_1| \geq {1/4}}
\left[\partial_{\xi_1}^{k_1} \phi (\xi_1)\right]\MF_1{\varphi_{k_2}}\left(\frac{s_1 \xi_1}{r_1},r_2\lambda_2
\right) e^{i s_1\lambda_1 \xi_1}d\xi_1\right|\\
&\qquad\leq (s_1 \lambda_1)^{-N_1} \sum_{k_1 + k_2 =N_1} \left( \frac{s}{r_1}\right)^{1+ k_2}|\eta(r_2\lambda_2)| \int_{|\xi_1| \geq {1/4}}
\left|\MF_1{\varphi_{k_2}}\left(\frac{s_1\xi_1}{r_1},r_2\lambda_2\right)
\right| d\xi_1 \\
&\qquad\lesssim  (s_1 \lambda_1)^{-N_1} \sum_{k_1 + k_2 =N_1} \left( \frac{s_1}{r_1}\right)^{1+ k_2-\sigma_1 +\varepsilon + {1/2}} |\eta(r_2\lambda_2)| \\
&\qquad\qquad \times\int_{|\xi_1| \geq {1/4}}
\left(1+\left|\frac{s_1 \xi_1}{r_1}\right|\right)^{\sigma_1 -
(\varepsilon + {1/2}) }\left|\MF_1{\varphi_{k_2}}\left(\frac{s_1 \xi_1}{r_1},r_2\lambda_2
\right) \right|d\xi_1\\
&\qquad\leq  (s_1 \lambda_1)^{-N_1} \sum_{k_1 + k_2 =N_1} \left( \frac{s_1}{r_1}\right)^{k_2 +
\varepsilon + {1/2} - \sigma_1 }|\eta(r_2\lambda_2)| \\
&\qquad\qquad \times \left[\int_{|\xi_1| \geq {1/4}}
\left|\left(1+\left|\frac{s_1\xi_1}{r_1}\right|\right)^{\sigma_1 }\MF_1{\varphi_{k_2}}\left(\frac{s_1\xi_1}{r_1},r_2\lambda_2
\right)\right|^{2}  d\left(\frac{s_1\xi_1}{r_1}\right)\right]^{1/2}\\
&\qquad\qquad \times \left[\int_{\R } \left(1+ \left| \frac{s_1\xi_1}{r_1}\right| \right)^{-1-2\varepsilon} d\left(\frac{s_1 \xi_1}{r_1}\right)\right]^{1/2}\\
&\qquad\lesssim  (s_1\lambda_1)^{-N_1} \sum_{k_1 + k_2 =N_1} \left( \frac{s_1}{r_1}\right)^{k_2+\varepsilon + {1/2} -\sigma_1 } |\eta(r_2\lambda_2)| \\
&\qquad\qquad \times \left[\int_{\R }
\left|\left(1+\left|\xi_1\right|\right)^{\sigma_1 }\MF_1{\varphi_{k_2}}(\xi_1,r_2\lambda_2)\right|^{2}  d\xi_1\right]^{1/2}.
\end{align*}
Since $W^{\varepsilon+1/2, 2}(\R ) \subset L^{\infty}(\R )$,
\begin{align*}
\left|\MF_1{\varphi_{k_2}}\left( \xi_1,r_2\lambda_2
\right)\right|^2 \lesssim \|\MF_1{\varphi_{k_2}}(\xi_1, \cdot)\|_{L^{\infty}(\R )}^2
 \lesssim \int_{\R }\big|(1 + |\xi_2|)^{\varepsilon + {1/2}}  \MF{\varphi_{k_2}}(\xi_1, \xi_2)\big|^2 d\xi_2.
\end{align*}
Therefore, it follows from Fubini's theorem that
\begin{align*}
&|F_{s_1,r_1,r_2} (\lambda_1, \lambda_2) |  \\
&\qquad\lesssim  (s_1\lambda_1)^{-N_1} \sum_{k_1 + k_2 =N_1} \left( \frac{s_1}{r_1}\right)^{k_2+\varepsilon + {1/2} -\sigma_1 }|\eta(r_2\lambda_2)| \\
&\hspace{4cm} \times  \left[\int_{\R \times \R }
\left|\left(1+\left|\xi_1\right|\right)^{\sigma_1 }(1+|\xi_2|)^{\varepsilon + {1/2}}\MF{\varphi_{k_2}}(\xi_1, \xi_2)\right|^{2}  d\xi_1d\xi_2\right]^{1/2}\\
&\qquad\lesssim (s_1 \lambda_1)^{-N_1} \left(1+\frac{s_1}{r_1} \right)^{N_1 + \varepsilon + {1/2} -\sigma_1} (1 +r_2\lambda_2)^{-N_2} \|\varphi_{k_2}\|_{W^{(\sigma_1, \varepsilon+1/2),2}(\R  \times \R )}\\
&\qquad\lesssim \|m\|_{W^{(\sigma_1, \varepsilon+1/2),2}(\R \times \R )}
\left[(1+s_1 \lambda_1)^{-N_1}\left( 1+ \frac{s_1}{r_1} \right)^{(N_1 + \varepsilon + {1/2})-\sigma_1}\right]
(1 +r_2\lambda_2)^{-N_2}.
\end{align*}

The proof of (iii) is parallel to that of (ii), and we omit the details.

To prove (iv), we consider sixteen cases depending on whether each of $\lambda_1 s_1, s_1/r_1, \lambda_2 s_2$ and $s_2/r_2$ is small or large.
Each of these cases can be handled by using a similar argument to that in the proof of (ii) and we omit the details here.
\end{proof}

If $m(\cdot,\cdot)$ has compact support with respect to the first variable only, then the following result holds;
its proof is similar to that of Lemma \ref{elem} (ii) and so will be omitted.

\begin{lemma} \label{elem1}
Suppose $m: [0,\infty) \times [0,\infty) \to \C $ is a function with
$\supp m \subset [1/4, 4] \times [0,\infty)$.
Let $\phi$ be a smooth cut-off function on $\R $ such that
\begin{align*}
\phi(\xi) =
\begin{cases}
0& \text{if  $|\xi| \leq 1/4$} \\
1& \text{if  $|\xi| \geq 1/2$}.
\end{cases}
\end{align*}
Set $\psi(\lambda_1, \lambda_2) := m(\lambda_1^2,\lambda_2^2)$ and define $F_{s_1,r_1,r_2}(\lambda_1, \lambda_2)$ by the formula
\begin{align*}
\MF_1F_{s_1,r_1} (\xi_1, \lambda_2) = \phi\left(\frac{\xi_1}{s_1}\right) \frac{1}{r_1}\MF_1{\psi}\left(\frac{\xi_1}{r_1},  \lambda_2\right),
\end{align*}
where $\MF_1$ is the Fourier transform with respect to the first variable.
Then, for any $\varepsilon>0$, any $N_1, N_2 \in \N$ and any $\sigma_1 \in \R $,
there exists a constant $C_{\varepsilon, \sigma_1, N_1,N_2}$ such that
\begin{align*}
|F_{s_1,r_1} (\lambda_1, \lambda_2)|  &\leq
C_{\varepsilon, \sigma_1, N_1,N_2}\|m\|_{W^{\sigma_1,2}L^\infty(\R \times \R )}  \left[(1+\lambda_1 s_1)^{-N_1}\left( 1+ \frac{s_1}{r_1} \right)^{(N_1 + \varepsilon + {1/2})-\sigma_1}\right]
\end{align*}
for all $r_1,s_1, \lambda_1,\lambda_2 >0$.  Here $\|m\|_{W^{\sigma_1, 2}L^\infty (\R  \times \R )}=
\sup_{\lambda_2 \in \R } \|m(\cdot, \lambda_2)\|_{W^{\sigma_1, 2}(\R )}$.
\end{lemma}

The following result concerning finite speed propagation was proved by Melrose \cite{Mel}.
\begin{lemma} \label{lem}
For $k=1,2$, there exists a constant $\kappa_k$ such that
\begin{align*}
\supp\left( K_{\cos (t\Box_b^{(k)})}\right) \subset \{(x,y) \in M^{(k)} \times M^{(k)}: \rho^{(k)}(x,y) \leq \kappa_k t\}.
\end{align*}
\end{lemma}

This result has the following bi-parameter analogue.
\begin{proposition} \label{fsp}
Let $F$ be a function on $\R \times \R $ such that $F(-\lambda_1, -\lambda_2) = F(\lambda_1, \lambda_2)$.
\\[2pt]
(i)\enspace
If $\supp \MF{F} \subset [-r_1, r_1] \times [-r_2, r_2]$, then
\begin{align*}
\supp K_{F(\sqrt{\Box_b^{(1)}}, \sqrt{\Box_b^{(2)}})}\subset \{(x_1, x_2,y_1,y_2):
\rho^{(1)}(x_1, y_1) \leq \kappa_1 r_1 \mbox{ and } \rho^{(2)}(x_2, y_2) \leq \kappa_2 r_2\}.
\end{align*}
\\[2pt]
(ii)\enspace
If $\supp \MF_1{F} (\cdot, \lambda_2)\subset [-r_1, r_1]$ for all $\lambda_2 \in \R $, then
\begin{align*}
\supp K_{F(\sqrt{\Box_b^{(1)}}, \sqrt{\Box_b^{(2)}})}\subset \{(x_1, x_2,y_1,y_2):
\rho^{(1)}(x_1, y_1) \leq \kappa_1 r_1\}.
\end{align*}
\\[2pt]
(iii)\enspace
If $\supp \MF_2{F} (\lambda_1, \cdot)\subset [-r_2, r_2]$ for all $\lambda_1 \in \R $, then
\begin{align*}
\supp K_{F(\sqrt{\Box_b^{(1)}}, \sqrt{\Box_b^{(2)}})}\subset \{(x_1, x_2,y_1,y_2):
\rho^{(2)}(x_2, y_2) \leq \kappa_2 r_2\}.
\end{align*}
\end{proposition}

\begin{proof}
We prove only (ii) since other assertions can be proved similarly.
Since $F$ is even, by the Fourier inversion formula,
\begin{align*}
F(\lambda_1, \lambda_2)
 =\int_{\R } \MF_1{F} (\xi_1, \lambda_2)e^{i \lambda_1 \xi_1} d\xi_1
 =\int_{-r_1}^{r_1} \MF_1{F} (\xi_1, \lambda_2)\cos(\lambda_1 \xi_1) d\xi_1.
\end{align*}
Hence
\begin{align*}
F(\Box_b^{(1)}, \Box_b^{(2)})
 =\int_{-r_1}^{r_1} \MF_1{F} (\xi_1, \Box_b^{(2)} )\cos(\xi_1 \Box_b^{(1)} )d\xi_1.
\end{align*}
It follows that
\begin{align*}
K_{F(\Box_b^{(1)}, \Box_b^{(2)})}(x_1, x_2, y_1, y_2)
 =\int_{-r_1}^{r_1} K_{\MF_1{F} (\xi_1, \Box_b^{(2)})}(x_2, y_2)K_{\cos(\xi_1\Box_b^1)}(x_1,y_1) d\xi_1.
\end{align*}
By Lemma \ref{lem},
\begin{align*}
\supp K_{\cos(\xi_1\Box_b^{(1)})} \subset \{(x_1, y_1): \rho^{(1)}(x_1, y_1) \leq \kappa_1 r_1 \}
 \end{align*}
for all $\xi_1 \in [-r_1, r_1]$, which yields the assertion (ii).
\end{proof}

We will need the following result. See \cite[Proposition 4.6]{Str1}.
\begin{lemma} \label{bes}
Let $k \in \{1,2\}$. For every multi-index $\alpha \in \mathcal{I}$ and every $\sigma > Q_k + 2 |\alpha|$,
there exists a constant $C_{k, \alpha, \sigma}$ such that for all $t >0$,
\begin{align} \label{087}
\big\|(X_{\alpha}^{(k)})_x K_{(I+t\hatbox_b^{(k)})^{-\sigma/4}}(x,\cdot) \big\|_{L^{2}(M^{(k)})} \leq C_{k, \alpha, m} \frac{\sqrt{t}^{-|\alpha|}}{V^{(k)}(x,\sqrt{t})^{1/2}}.
\end{align}
\end{lemma}
\begin{remark} \label{rmk}
Inequality \eqref{087} can be rewritten as
\begin{align*}
\big\|(X_{\alpha}^{(k)})_x K_{(I+t\sqrt{\hatbox_b^{(k)}})^{-\sigma/2}}(x,\cdot)
\big\|_{L^{2}(M^{(k)})} \leq C_{k, \alpha, m} \frac{t^{-|\alpha|}}{V^{(k)}(x,t)^{1/2}}.
\end{align*}
Indeed, since
\begin{align*}
\left(I+t\sqrt{\hatbox_b^{(k)}}\right)^{-\sigma/2} =\phi(\Box_b^{(k)})\left(I+t^{2} \hatbox_b^{(k)}\right)^{-\sigma/4},
\end{align*}
where $\phi(\lambda)=(1+t\sqrt{\lambda})^{-\sigma/2} (1 + t^2\lambda)^{\sigma/4}$,
\begin{align*}
&\big\|(X_{\alpha}^{(k)})_x K_{\bigl(I+t\sqrt{\hatbox_b^{(k)}}\bigr)^{-\sigma/2}}(x,\cdot) \big\|_{L^{2}(M^{(k)})}\\
&= \big\|(X_{\alpha}^{(k)})_x K_{\bigl(I+t\sqrt{\hatbox_b^{(k)}}\bigr)^{-\sigma/2}}(\cdot,x) \big\|_{L^{2}(M^{(k)})}\\
& = \big\|(X_{\alpha}^{(k)})_x K_{\phi(\hatbox_b^{(k)})\left(I+t^{2} \hatbox_b^{(k)}\right)^{-\sigma/4}}(\cdot,x) \big\|_{L^{2}(M^{(k)})}\\
&= \left\|(X_{\alpha}^{(k)})_x
\int K_{\phi(\Box_b^{(k)}) }(\cdot,z) K_{\left(I+t^{2} \hatbox_b^{(k)}\right)^{-\sigma/4}}(z,x) \,dz\right\|_{L^{2}(M^{(k)})}\\
&=\left\|\int K_{\phi(\Box_b^{(k)}) }(\cdot, z)\left[ (X_{\alpha}^{(k)})_xK_{\left(I+t^{2} \hatbox_b^{(k)}\right)^{-\sigma/4}}(z,x) \right] \,dz\right\|_{L^{2}(M^{(k)})}\\
&=\left\| \phi(\Box_b^{(k)})\left[(X_{\alpha}^{(k)})_x K_{\left(I+t^{2} \hatbox_b^{(k)}\right)^{-\sigma/4}}(\cdot,x) \right] \right\|_{L^{2}(M^{(k)})}\\
&\lesssim \left\| (X_{\alpha}^{(k)})_x \left[K_{\left(I+t^{2} \hatbox_b^{(k)}\right)^{-\sigma/4}}(\cdot,x) \right] \right\|_{L^{2}(M^{(k)})}\\
&= \left\| (X_{\alpha}^{(k)})_x \left[K_{\left(I+t^{2} \hatbox_b^{(k)}\right)^{-\sigma/4}}(x,\cdot)
 \right] \right\|_{L^{2}(M^{(k)})}\\
& \leq C_{k, \alpha, m} \frac{t^{-|\alpha|}}{V^{(k)}(x,t)^{1/2}},
\end{align*}
where we used the fact that $K_{\phi(\Box_b^{(k)})} (x,y) = K_{\phi(\Box_b^{(k)})} (y,x)\overline{\phantom{m}}$.
\end{remark}

Our next two propositions establish pointwise kernel estimates for the multipliers, when the multiplier function is smooth enough and has compact support.
The first proposition is for one group of variables and the second  is for
both groups of variables.

\begin{proposition} \label{ker1}
Suppose that $m:[0,\infty) \times [0,\infty) \to \C $ is supported in
$[1/4, 4] \times [0, \infty)$ and
$\|m\|_{W^{\sigma_1,2}L^\infty(\R \times \R )}<\infty$ for some $\sigma_1 \in \R $.
Then, for any $\varepsilon>0$, any $\alpha, \beta\in \mathcal{I}$,
any $N_1 \in \N$ such that $N_1 > (|\alpha| \vee |\beta|) + {Q_1}/{2}$,
there exists a constant $C_{\varepsilon, \alpha, \beta, N_1}$
such that
\begin{align*}
&\left|\iint (X_{\alpha}^{(1)})_{x_1} (X_{\beta}^{(1)})_{y_1} K_{m(r_1^2 \Box_b^{(1)}, \Box_b^{(2)})}(x_1,y_1,x_2, y_2)f_2(x_2)g_2(y_2)
\,d\mu^{2}(x_2) \,d\mu^{2}(y_2)\right|\\
&\qquad\leq C_{\varepsilon,\alpha, \beta, \sigma_1,N_1} \|m\|_{W^{\sigma_1,2}L^\infty(\R \times \R )}\|f_2\|_2\|g_2\|_2 \\
&\qquad\qquad \times \left(1+\frac{\rho^{(1)}(x_1,y_1)}{r_1} \right)^{N_1 +\varepsilon + {1/2} -\sigma_1}\frac{(r_1 \vee \rho^{(1)}(x_1, y_1))^{-|\alpha| -|\beta|}}{V^{(1)}(x_1, \rho^{(1)}(x_1,y_1) +r_1)}
\end{align*}
for all $x_1, y_1 \in M^{(1)}$.
\end{proposition}

\begin{proof}
Fix arbitrary $x_1, y_1 \in M^{(1)}$. We prove the assertion by considering four cases.

\subsubsection*{Case 1.} $\rho^{(1)}({x}_1, {y}_1) \geq r_1$.
Let $\psi(\lambda_1, \lambda_2) := m(\lambda_1^2, \lambda_2^2)$, $\psi_{r_1}(\lambda_1, \lambda_2):=\psi(r_1 \lambda_1,
\lambda_2)$ and define $F_{s_1,r_1}(\lambda_1, \lambda_2)$ by the formula
\begin{align*}
\MF{ F_{s_1, r_1}} (\xi_1, \xi_2) &=\phi\left(\frac{\xi_1}{s_1}\right)
 \frac{1}{r_1}\MF{\psi} \left(\frac{\xi_1}{r_1}, \xi_2\right)
  = \phi\left(\frac{\xi_1}{s_1}\right)
 \MF{\psi_{r_1}} (\xi_1, \xi_2),
\end{align*}
where $\phi$ is as in Lemma \ref{elem1}.
Let $J_{s_1,s_2,r_1} (\lambda_1, \lambda_2)$ be a complex-valued function such that
\begin{align*}
J_{s_1,r_1} (\lambda_1, \lambda_2)^2= F_{s_1,r_1} (\lambda_1, \lambda_2)
\end{align*}
(any choice of the square root will do).
Set $s_1= {\rho^{(1)}({x}_1,{y}_1)}/{(4\kappa_1)}$.
Since
\begin{align*}
\supp \bigl(\MF{F_{s_1,r_1}}(\cdot, \cdot) - \MF{\psi_{r_1}}(\cdot, \cdot)\bigr) \subset [0,2s_1]\times [0,\infty)
= \left[0,\frac{\rho^{(1)}({x}_1,{y}_1)}{2\kappa_1}\right] \times [0,\infty),
\end{align*}
it follows from the finite speed propagation (Proposition \ref{fsp}) that
\begin{align*}
&\supp  \left(K_{F_{s_1,s_2,r_1}(\sqrt{\Box_b^{(1)}}, \sqrt{\Box_b^{(2)}})}(\cdot, \cdot, \cdot, \cdot) - K_{\psi_{r_1}(\sqrt{\Box_b^{(1)}}, \sqrt{\Box_b^{(2)}})}(\cdot, \cdot,\cdot,\cdot)\right) \\
&\qquad \subset \left\{(u_1, v_1, x_2, y_2):  \rho^{(1)}(u_1,v_1) \leq \frac{\rho^{(1)}({x}_1, {y}_1)}{2}\right\}.
\end{align*}
Hence we can write
\begin{align*}
&\iint (X_{\alpha}^{(1)})_{x_1} (X_{\beta}^{(1)})_{y_1} K_{m(r_1^2 \Box_b^{(1)}, \Box_b^{(2)})}(x_1,y_1,x_2, y_2)f_2(x_2)g_2(y_2)\,d\mu^{2}(x_2)\,d\mu^{2}(y_2)\\
&\qquad=\iint (X_{\alpha}^{(1)})_{x_1} (X_{\beta}^{(1)})_{y_1} K_{\psi_{r_1}(\sqrt{\Box_b^{(1)}}, \sqrt{\Box_b^{(2)}})}(x_1,y_1,x_2, y_2)f_2(x_2)g_2(y_2)\,d\mu^{2}(x_2)\,d\mu^{2}(y_2)\\
&\qquad=\iint (X_{\alpha}^{(1)})_{x_1} (X_{\beta}^{(1)})_{y_1} K_{F_{s_1,r_1}(\sqrt{\Box_b^{(1)}}, \sqrt{\Box_b^{(2)}})}(x_1,y_1,x_2, y_2)f_2(x_2)g_2(y_2)\,d\mu^{2}(x_2)\,d\mu^{2}(y_2) \\
&\qquad=\iint (X_{\alpha}^{(1)})_{x_1} (X_{\beta}^{(1)})_{y_1} K_{F_{s_1,r_1}(\sqrt{\hatbox_b^{(1)}}, \sqrt{\Box_b^{(2)}})}(x_1,y_1,x_2, y_2)f_2(x_2)g_2(y_2)\,d\mu^{2}(x_2)\,d\mu^{2}(y_2)\\
&\qquad\qquad+ \iint (X_{\alpha}^{(1)})_{x_1} (X_{\beta}^{(1)})_{y_1} K_{\pi^{(1)} \otimes F_{s_1,r_1}(0, \sqrt{\Box_b^{(2)}})}(x_1,y_1,x_2, y_2)f_2(x_2)g_2(y_2)\,d\mu^{2}(x_2)\,d\mu^{2}(y_2)\\
&\qquad:= I_1 + I_2.
\end{align*}

Since $\pi^{(1)}$ is an NIS operator,
\begin{align*}
|(X_{\alpha}^{(1)})_{x_1}(X_{\beta}^{(1)})_{y_1} K_{\pi^{(1)}}(x_1,y_1)|
&\leq C_{\alpha, \beta} \frac{\rho^{(1)}(x_1,y_1)^{-|\alpha|-|\beta|}}{V^{(1)}(x_1, \rho^{(1)}(x_1,y_1))} \\
&\sim \frac{\rho^{(1)}(x_1,y_1)^{-|\alpha|-|\beta|}}{V^{(1)}(x_1, \rho^{(1)}(x_1,y_1)+r_1)} .
\end{align*}

Therefore, from Lemma \ref{elem1},
\begin{align*}
|I_2| &\lesssim  C_{\alpha, \beta} \frac{\rho^{(1)}({x}_1,{y}_1)^{-|\alpha|-|\beta|}}{V^{(1)}
({x}_1, \rho^{(1)}({x}_1,{y}_1)+r_1)}
\|F_{s_1,s_2,r_1}(0, \sqrt{\Box_b^{(2)}})\|_{2\to 2} \|f_2\|_2\|g_2\|_2\\
&\leq  C_{\varepsilon,\alpha, \beta, \sigma_1,N_1} \|m\|_{W^{\sigma_1,2}L^\infty(\R \times \R )}\|f_2\|_2\|g_2\|_2 \\
 &\quad \times \left(1+\frac{\rho^{(1)}({x}_1,{y}_1)}{r_1} \right)^{N_1 +\varepsilon + {1/2} -\sigma_1}
 \frac{(r_1 \vee \rho^{(1)}({x}_1, {y}_1))^{-|\alpha| -|\beta|}}{V^{(1)}({x}_1,
 \rho^{(1)}({x}_1,{y}_1) +r_1)}.
\end{align*}

Next we estimate term $I_1$. Using the fact that
\begin{align*}
F_{s_1,r_1}(\sqrt{\hatbox_b^{(1)}},\sqrt{\Box_b^{(2)}}) = J_{s_1,r_1}(\sqrt{\hatbox_b^{(1)}},\sqrt{\Box_b^{(2)}}) \circ J_{s_1,r_1}(\sqrt{\hatbox_b^{(1)}},\sqrt{\Box_b^{(2)}})
\end{align*}
 and the Cauchy--Schwarz inequality, we see that
\begin{align*}
|I_1|&=\left|\iint (X_{\alpha}^{(1)})_{x_1} (X_{\beta}^{(1)})_{y_1} K_{F_{s_1,r_1}(\sqrt{\hatbox_b^{(1)}}, \sqrt{\Box_b^{(2)}})}(x_1,y_1,x_2, y_2)f_2(x_2)g_2(y_2)\,d\mu^{2}(x_2)\,d\mu^{2}(y_2)\right|\\
& = \left|\iint (X_{\alpha}^{(1)})_{x_1} (X_{\beta}^{(1)})_{y_1} \iint  K_{J_{s_1,r_1}
(\sqrt{\hatbox_b^{(1)}}, \sqrt{{\Box}_b^{(2)}})}(x_1,z_1,x_2,z_2)\right.\\
 &\qquad
 \times \left. K_{J_{s_1, r_1}(\sqrt{\hatbox_b^{(1)}}, \sqrt{{\Box}_b^{(2)})}}(z_1,y_1,z_2,y_2)
\,d\mu^{1}(z_1) \,d\mu^{2}(z_2)f_2(x_2)g_2(y_2)\,d\mu^{2}(x_2)\,d\mu^{2}(y_2)\right|\\
& \leq \left\|\int (X_{\alpha}^{(1)})_{x_1}   K_{J_{s_1,r_1}
(\sqrt{\hatbox_b^{(1)}}, \sqrt{{\Box}_b^{(2)}})}(x_1, \cdot, x_2,\cdot)f_2(x_2)\,d\mu^{2}(x_2)
 \right\|_{L^{2}(M^{(1)} \times M^{(2)})} \\
 &\quad \times \left\|\int (X_{\beta}^{(1)})_{y_1}   K_{J_{s_1,r_1}
(\sqrt{\hatbox_b^{(1)}}, \sqrt{{\Box}_b^{(2)}})}(\cdot,y_1,\cdot,y_2)g_2(y_2)
\,d\mu^{2}(y_2) \right\|_{L^{2}(M^{(1)} \times M^{(2)})} \\
 &=:I_{11}\times I_{12}.
 \end{align*}
We only estimate term $I_{11}$, as $I_{12}$ is similar.
Set $V^{(1)}_{s_1}(x_1)=V^{(1)}({x}_1,s_1)^{1/2}$.
We write
\begin{align*}
&V^{(1)}({x}_1,s_1)^{1/2}I_{11}\\
&= \sup_{\|h\|_2=1} \biggl|V^{(1)}_{s_1}(x_1)\iiint (X_{\alpha}^{(1)})_{x_1}   K_{J_{s_1,r_1}
(\sqrt{\hatbox_b^{(1)}}, \sqrt{{\Box}_b^{(2)}})}(x_1,z_1,x_2,z_2)f_2(x_2)\,d\mu^2(x_2) \\
&\hspace{5cm} \times h(z_1,z_2)\,d\mu^{1}(z_1) \,d\mu^{2}(z_2)\biggr|\\
&\leq \sup_{\|h\|_2=1,\|g_1\|_1=1} \biggl|\iiiint V^{(1)}_{s_1}(u_1)(X_{\alpha}^{(1)})_{u_1}   K_{J_{s_1,r_1}
(\sqrt{\hatbox_b^{(1)}}, \sqrt{{\Box}_b^{(2)}})}(u_1,z_1,x_2,z_2)f_2(x_2) \,d\mu^{2}(x_2) \\
&\hspace{5cm}\times h(z_1,z_2)g_1(u_1)\,d\mu^{1}(u_1)  \,d\mu^{1}(z_1) \,d\mu^{2}(z_2)\biggr|\\
&= \sup_{\|h\|_2=1,\|g_1\|_1=1} \biggl|\iiiint    K_{J_{s_1,r_1}
(\sqrt{\hatbox_b^{(1)}}, \sqrt{{\Box}_b^{(2)}})}(u_1,z_1,x_2,z_2)f_2(x_2)\,d\mu(x_2) h(z_1,z_2)\,d\mu^{1}(z_1) \,d\mu^{2}(z_2) \\
&\hspace{5cm}\times(X_{\alpha}^{(1)})_{u_1}V^{(1)}_{s_1}(u_1)g_1(u_1)\,d\mu^{1}(u_1) \biggr|\\
&=\sup_{\|h\|_2=1,\|g_1\|_1=1} \biggl|\iint    J_{s_1,r_1}
(\sqrt{\hatbox_b^{(1)}}, \sqrt{{\Box}_b^{(2)}})h(u_1,x_2) f_2(x_2)(X_{\alpha}^{(1)})_{u_1}V^{(1)}_{s_1}(u_1) \\
&\hspace{5cm} \times g_1(u_1)\,d\mu^{1}(u_1) \,d\mu^{2}(x_2)\biggr|\\
&= \sup_{\|h\|_2=1,\|g_1\|_1=1} \biggl|\iint  G(\sqrt{\hatbox_b^{(1)}}, \sqrt{\Box_b^{(2)}}) \left(I + s_1 \sqrt{\hatbox_b^{(1)}}\right)^{-N_1 /2}  \otimes I h(u_1,x_2)  \\
&\hspace{5cm}\times f_2(x_2)(X_{\alpha}^{(1)})_{u_1}V^{(1)}_{s_1}(u_1)
g_1(u_1)\,d\mu^{1}(u_1) \,d\mu^{2}(x_2)\biggr|\\
&= \sup_{\|h\|_2=1,\|g_1\|_1=1} \biggl|\iint    G(\sqrt{\hatbox_b^{(1)}}, \sqrt{\Box_b^{(2)}})
h(u_1,x_2) \\
&\hspace{2cm}\times \left(I + s_1 \sqrt{\hatbox_b^{(1)}}\right)^{-N_1 /2}  \otimes I \left[f_2(x_2)(X_{\alpha}^{(1)})_{u_1}V^{(1)}_{s_1}(u_1)g_1(u_1)\right]\,d\mu^{1}(u_1) \,d\mu^{2}(x_2)\biggr|
\end{align*}
where
\begin{align*}
G(\lambda_1, \lambda_2) = (1 + s_1 \lambda_1)^{N_1 /2} J_{s_1, r_1}(\lambda_1, \lambda_2),
\end{align*}
and for the third step we used the estimate
 \begin{align*}
\zeta(x_1) \leq \|\zeta\|_{\infty} = \sup_{\|\theta\|_{1}=1} \int_{M^{(1)}} \zeta(u_1) \theta(u_1) \,d\mu^{1}(u_1)
\qquad\forall x_1 \in M^{(1)}.
 \end{align*}
By Lemma \ref{elem1},
\begin{align*}
|G(\lambda_1, \lambda_2)|&= |(1 + s_1 \lambda_1)^{N_1 /2}|| J_{s_1, r_1}(\lambda_1, \lambda_2)|\\
& \lesssim |(1 + s_1 \lambda_1)^{N_1 /2}| | F_{s_1, r_1}(\lambda_1, \lambda_2)|^{1/2}\\
& \lesssim \left( 1+ \frac{s_1}{r_1} \right)^{({N_1} + +{\varepsilon} + 1/2)/2-{\sigma_1/2}}\|m\|^{1/2}_{W^{\sigma_1,2}L^\infty(\R \times \R )}.
\end{align*}
From this and the Cauchy--Schwarz inequality, it follows that
\begin{align*}
&V^{(1)}({x}_1,s_1)^{1/2}I_{11}\\
&\leq C \sup_{\|h\|_2=1,\|g_1\|_1=1}\left\|  G(\sqrt{\hatbox_b^{(1)}}, \sqrt{\Box_b^{(2)}})h\right\|_2 \\
&\hspace{5cm}\left\|\left(I + s_1 \sqrt{\hatbox_b^{(1)}}\right)^{-N_1 /2}(X_{\alpha}^{(1)})_{u_1}V^{(1)}_{s_1}(u_1)g_1(u_1)\right\|_2 \|f_2\|_2\\
&\leq C \sup_{\|g_1\|_1=1}  \|G\|_\infty \left\|\left(I + s_1 \sqrt{\hatbox_b^{(1)}}\right)^{-N_1 /2}(X_{\alpha}^{(1)})_{u_1}V^{(1)}_{s_1}(u_1)g_1(u_1)\right\|_2 \|f_2\|_2\\
&\leq C \left( 1+ \frac{s_1}{r_1} \right)^{({N_1/2} + {1/4}+{\varepsilon/2})-{\sigma_1/2}}
 \|m\|^{1/2}_{W^{\sigma_1,2}L^\infty(\R \times \R )}\|f_2\|_2  \\
&\hspace{5cm} \times \sup_{\|g_1\|_1=1}  \left\|\left(I + s_1 \sqrt{\hatbox_b^{(1)}}\right)^{-N_1 /2}(X_{\alpha}^{(1)})_{u_1}V^{(1)}_{s_1}(u_1)g_1(u_1)\right\|_2.
\end{align*}
Hence, from Lemma \ref{bes} and Remark \ref{rmk},
\begin{align*}
&\left\|\left(I + s_1 \sqrt{\hatbox_b^{(1)}}\right)^{-N_1 /2} (X_{\alpha}^{(1)})_{u_1}V^{(1)}_{s_1}(u_1)g_1(u_1)\right\|_2\\
&\qquad\leq \sup_{\|f_1\|_2=1} \left|\int \left(I + s_1 \sqrt{\hatbox_b^{(1)}}\right)^{-N_1 /2}(X_{\alpha}^{(1)})_{u_1}V^{(1)}_{s_1}(u_1)g_1(u_1) f_1(u_1)\,d\mu^{1}(u_1)  \right|\\
&\qquad=  \sup_{\|f_1\|_2=1} \left|\int g_1(u_1) V^{(1)}_{s_1}(u_1)(X_{\alpha}^{(1)})_{u_1} \left(I + s_1 \sqrt{\hatbox_b^{(1)}}\right)^{-N_1 /2}f_1(u_1)\,d\mu^{1}(u_1)  \right|\\
&\qquad\leq \sup_{\|f_1\|_2=1} \|g_1\|_1\sup_{u_1} V^{(1)}_{s_1}(u_1)\left|(X_{\alpha}^{(1)})_{u_1}\int K_{\left(I + s_1 \sqrt{\hatbox_b^{(1)}}\right)^{-N_1 /2}}(u_1,y_1)f_1(y_1)\,d\mu^{1}(y_1) \right|\\
&\qquad\leq C \sup_{\|f_1\|_2=1}\|g_1 \|_1\sup_{u_1} V^{(1)}_{s_1}(u_1)\left\|(X_{\alpha}^{(1)})_{u_1} K_{\left(I + s_1 \sqrt{\hatbox_b^{(1)}}\right)^{-N_1 /2}}(u_1,\cdot)\,d\mu^{1}(u_1) \right\|_2\|f_1\|_2\\
&\qquad\leq C \|g_1\|_1 s_1^{-|\alpha|}.
\end{align*}
Combining the two estimates above, we deduce that
 \begin{align*}
 I_{11}\leq
 C \left( 1+ \frac{s_1}{r_1} \right)^{({N_1/2} + {1/4}+{\varepsilon/2})-{\sigma_1/2}}
 \|m\|^{1/2}_{W^{\sigma_1,2}L^\infty(\R \times \R )}\|f_2\|_2 s_1^{-|\alpha|}V^{(1)}({x}_1,s_1)^{-1/2}.
\end{align*}
Similarly,
\begin{align*}
I_{12}\leq
 C \left( 1+ \frac{s_1}{r_1} \right)^{({N_1/2} + {1/4}+{\varepsilon/2})-{\sigma_1/2}}\|m\|^{1/2}_{W^{\sigma_1,2}L^\infty(\R \times \R )}
 \|g_2\|_2 s_1^{-|\beta|}V^{(1)}({y}_1,s_1)^{-1/2}.
\end{align*}
Consequently,
\begin{align*}
|I_1|  & \lesssim C \|m\|_{W^{\sigma_1,2}L^\infty(\R \times \R )}\|f_2\|_2\|g_2\|_2
  \left(1+\frac{\rho^{(1)}(x_1,y_1)}{r_1} \right)^{N_1 +\varepsilon + {1/2} -\sigma_1}\frac{(r_1 \vee \rho^{(1)}(x_1, y_1))^{-|\alpha| -|\beta|}}{V^{(1)}(x_1, \rho^{(1)}(x_1,y_1) +r_1)}.
 \end{align*}

\subsubsection*{Case 2.} $\rho^{(1)}({x}_1, {y}_1) \leq r_1$.
Let $\psi(\lambda_1, \lambda_2) := m(\lambda_1^2, \lambda_2^2)$, $\psi_{r_1}(\lambda_1, \lambda_2):=\psi(r_1 \lambda_1, \lambda_2)$ and let $J_{r_1} (\lambda_1, \lambda_2)$ be any complex-valued function such that
\begin{align*}
J_{r_1} (\lambda_1, \lambda_2)^2=\psi(r_1 \lambda_1, \lambda_2)
\end{align*}
(any branch will do).
Since $J^2_{r_1} (0, \lambda_2)=\psi(0,
\lambda_2)=m(0,\lambda_2^2)=0$, it follows that
\begin{align*}
J_{r_1}(\sqrt{\Box_b^{(1)}}, \sqrt{\Box_b^{(2)}})=J_{r_1}(\sqrt{\hatbox_b^{(1)}}, \sqrt{\Box_b^{(2)}}).
\end{align*}
Hence we can write
\begin{align*}
&\iint (X_{\alpha}^{1})_{x_1} (X_{\beta}^{(1)})_{y_1} K_{m(r_1^2 \Box_b^{(1)}, \Box_b^{(2)})}(x_1,y_1,x_2, y_2)f_2(x_2)g_2(y_2)\,d\mu^{2}(x_2) \,d\mu^{2}(y_2) \\
&\qquad=\iint (X_{\alpha}^{(1)})_{x_1} (X_{\beta}^{(1)})_{y_1} K_{\psi_{r_1}(\sqrt{\Box_b^{(1)}}, \sqrt{\Box_b^{(2)}})}(x_1,y_1,x_2, y_2)f_2(x_2)g_2(y_2) \,d\mu^{2}(x_2) \,d\mu^{2}(y_2) \\
&\qquad=\iint (X_{\alpha}^{(1)})_{x_1} (X_{\beta}^{(1)})_{y_1} K_{J^2_{r_1}(\sqrt{\hatbox_b^{(1)}}, \sqrt{\Box_b^{(2)}})}(x_1,y_1,x_2, y_2)f_2(x_2)g_2(y_2)\,d\mu^{2}(x_2) \,d\mu^{2}(y_2)\\
&\qquad=:B,
\end{align*}
say.
From the Cauchy--Schwarz inequality, it follows that
\begin{align*}
B&=\left|\iint (X_{\alpha}^{(1)})_{x} (X_{\beta}^{(1)})_{y_1} K_{J^2_{r_1}(\sqrt{\hatbox_b^{(1)}}, \sqrt{\Box_b^{(2)}})}(x_1,y_1,x_2, y_2)f_2(x_2)g_2(y_2)\,d\mu^{2}(x_2) \,d\mu^{2}(y_2)\right|\\
& = \left|\iint (X_{\alpha}^{(1)})_{x} (X_{\beta}^{(1)})_{y_1} \int_{M^{(1)} \times M^{(2)}} K_{J_{r_1}
(\sqrt{\hatbox_b^{(1)}}, \sqrt{{\Box}_b^{(2)}})}(x_1,z_1,x_2,z_2)\right.\\
&\hspace{3cm}\times \left. K_{J_{r_1}(\sqrt{\hatbox_b^{(1)}}, \sqrt{{\Box}_b^{(2)})}}(z_1,y_1,z_2,y_2)
\,dz_1 \,dz_2f_2(x_2)g_2(y_2)\,d\mu^{2}(x_2) \,d\mu^{2}(y_2)\right|\\
& \leq \left\|\int (X_{\alpha}^{(1)})_{x_1}   K_{J_{r_1}
(\sqrt{\hatbox_b^{(1)}}, \sqrt{{\Box}_b^{(2)}})}(x_1,z_1,x_2,z_2)f_2(x_2)\,d\mu^{2}(x_2) \right\|_{L^{2}(M^{(1)} \times M^{(2)})} \\
 &\hspace{3cm} \times \left\|\int (X_{\beta}^{(1)})_{y_1}   K_{J_{r_1}
(\sqrt{\hatbox_b^{(1)}}, \sqrt{{\Box}_b^{(2)}})}(z_1,y_1,z_2,y_2)g_2(y_2)\,d\mu^{2}(y_2) \right\|_{L^{2}(M^{(1)} \times M^{(2)})} \\
 &=:B_{1}\times B_{2}.
 \end{align*}
Set $V^{(1)}_{r_1}(x_1)=V^{(1)}({x}_1,r_1)^{1/2}$.
We write
\begin{align*}
&V^{(1)}({x}_1,r_1)^{1/2}I_{11}\\
&\qquad= \sup_{\|h\|_2=1} \biggl|V^{(1)}_{r_1}(x_1)\iint (X_{\alpha}^{(1)})_{x_1}   K_{J_{r_1}
(\sqrt{\hatbox_b^{(1)}}, \sqrt{{\Box}_b^{(2)}})}(x_1,z_1,x_2,z_2) \\
&\qquad\hspace{7cm} f_2(x_2) \,dx_2 h(z_1,z_2) \,d\mu^{1}(z_1) \,d\mu^{2}(z_2)\biggr|\\
&\qquad\leq \sup_{\|h\|_2=1,\|g_1\|_1=1} \biggl|\iint V^{(1)}_{r_1}(x_1)(X_{\alpha}^{(1)})_{x_1}   K_{J_{r_1}
(\sqrt{\hatbox_b^{(1)}}, \sqrt{{\Box}_b^{(2)}})}(x_1,z_1,x_2,z_2)
f_2(x_2) \,d\mu^{2}(x_2)\\
&\qquad\hspace{7cm}  h(z_1,z_2) g_1(x_1)\,d\mu^{1}(x_1)  \,d\mu^{1}(z_1) \,d\mu^{2}(z_2)\biggr|\\
&\qquad= \sup_{\|h\|_2=1,\|g_1\|_1=1} \biggl|\iiint    K_{J_{r_1}
(\sqrt{\hatbox_b^{(1)}}, \sqrt{{\Box}_b^{(2)}})}(x_1,z_1,x_2,z_2)f_2(x_2) \,d\mu^{2}(x_2) h(z_1,z_2)\,d\mu^{1}(z_1) \\
&\qquad\hspace{7cm}  \,d\mu^{2}(z_2)(X_{\alpha}^{(1)})_{x_1}V^{(1)}_{r_1}(x_1)g_1(x_1)
\,d\mu^{1}(x_1) \biggr|\\
&\qquad= \sup_{\|h\|_2=1,\|g_1\|_1=1} \biggl|\iint    J_{r_1}
(\sqrt{\hatbox_b^{(1)}}, \sqrt{{\Box}_b^{(2)}})h(x_1,x_2) f_2(x_2)(X_{\alpha}^{(1)})_{x_1}V^{(1)}_{r_1}(x_1)\\
&\qquad\hspace{7cm} g_1(x_1)\,d\mu^{1}(x_1) \,d\mu^{2}(x_2)\biggr|\\
&\qquad= \sup_{\|h\|_2=1,\|g_1\|_1=1} \biggl|\iint G(\sqrt{\hatbox_b^{(1)}}, \sqrt{\Box_b^{(2)}}) \left(I + r_1 \sqrt{\hatbox_b^{(1)}}\right)^{-N_1 /2}  \otimes I h(x_1,x_2) f_2(x_2) \\
&\qquad\hspace{7cm}  (X_{\alpha}^{(1)})_{x_1}V^{(1)}_{r_1}(x_1)g_1(x_1)\,d\mu^{1}(x_1) \,d\mu^{2}(x_2)\biggr| \\
&\qquad= \sup_{\|h\|_2=1,\|g_1\|_1=1} \biggl|\iint    G(\sqrt{\hatbox_b^{(1)}}, \sqrt{\Box_b^{(2)}})
h(x_1,x_2) \\
&\qquad\hspace{4cm}  \times \left(I + r_1 \sqrt{\hatbox_b^{(1)}}\right)^{-N_1 /2}  \otimes I \left[f_2(x_2)(X_{\alpha}^{(1)})_{x_1}V^{(1)}_{r_1}(x_1)g_1(x_1)\right]\\
&\qquad\hspace{10cm} \times \,d\mu^{1}(x_1) \,d\mu^{2}(x_2)\biggr|,
\end{align*}
where
\begin{align*}
G(\lambda_1, \lambda_2) = (1 + r_1 \lambda_1)^{N_1 /2} J_{r_1}(\lambda_1, \lambda_2).
\end{align*}
Since the supports of $m$ and $\psi$ are compact,
\begin{align*}
|G(\lambda_1, \lambda_2)|\lesssim |(1 + r_1 \lambda_1)^{N_1 /2}| |\psi(r_1 \lambda_1,
\lambda_2)|^{1/2} \lesssim C\|m\|^{1/2}_\infty.
\end{align*}
From this and the Cauchy--Schwarz inequality, it follows that
\begin{align*}
&V^{(1)}({x}_1,r_1)^{1/2}B_{1}\\
&\qquad\leq C \sup_{\|h\|_2=1,\|g_1\|_1=1} \left\|  G(\sqrt{\hatbox_b^{(1)}}, \sqrt{\Box_b^{(2)}})h\right\|_2 \\
&\qquad\hspace{4cm}\left\|\left(I + r_1 \sqrt{\hatbox_b^{(1)}}\right)^{-N_1 /2}(X_{\alpha}^{(1)})_{x_1}V^{(1)}_{r_1}(x_1)g_1(x_1)\right\|_2 \|f_2\|_2\\
&\qquad\leq C \sup_{\|g_1\|_1=1}  \|G\|_\infty \left\|\left(I + r_1 \sqrt{\hatbox_b^{(1)}}\right)^{-N_1 /2}(X_{\alpha}^{(1)})_{x_1}V^{(1)}_{r_1}(x_1)g_1(x_1)\right\|_2 \|f_2\|_2\\
&\qquad\leq C \|m\|^{1/2}_\infty \|f_2\|_2 \sup_{\|g_1\|_1=1}  \left\|\left(I + r_1 \sqrt{\hatbox_b^{(1)}}\right)^{-N_1 /2}(X_{\alpha}^{(1)})_{x_1}V^{(1)}_{r_1}(x_1)g_1(x_1)\right\|_2.
\end{align*}
Together with Lemma \ref{bes} and Remark \ref{rmk}, this yields
\begin{align*}
&\left\|\left(I + r_1 \sqrt{\hatbox_b^{(1)}}\right)^{-N_1 /2}(X_{\alpha}^{(1)})_{x_1}V^{(1)}_{r_1}(x_1)g_1(x_1)\right\|_2\\
&\qquad\leq \sup_{\|f_1\|_2=1} \left|\int \left(I + r_1 \sqrt{\hatbox_b^{(1)}}\right)^{-N_1 /2}(X_{\alpha}^{(1)})_{x_1}V^{(1)}_{r_1}(x_1)g_1(x_1) f_1(x_1)\,d\mu^{1}(x_1)  \right|\\
&\qquad=  \sup_{\|f_1\|_2=1} \left|\int g_1(x_1) V^{(1)}_{r_1}(x_1)(X_{\alpha}^{(1)})_{x_1} \left(I + r_1 \sqrt{\hatbox_b^{(1)}}\right)^{-N_1 /2}f_1(x_1)\,d\mu^{1}(x_1) \right|\\
&\qquad\leq \sup_{\|f_1\|_2=1} \|g_1\|_1\sup_{x_1} V^{(1)}_{r_1}(x_1)\left|(X_{\alpha}^{(1)})_{x_1}\int K_{\left(I + r_1 \sqrt{\hatbox_b^{(1)}}\right)^{-N_1 /2}}(x_1,y_1)f_1(y_1)\,d\mu^{1}(y_1)\right|\\
&\qquad\leq C \sup_{\|f_1\|_2=1}\|g_1\|_1\sup_{x_1} V^{(1)}_{r_1}(x_1)\left\|X_{\alpha}^{(1)})_{x_1} K_{\left(I + r_1 \sqrt{\hatbox_b^{(1)}}\right)^{-N_1 /2}}(x_1,\cdot)\right\|_2\|f_1\|_2\\
&\qquad\leq C \|g_1\|_1 r_1^{-|\alpha|}.
\end{align*}
Combining the two estimates above shows that
 \begin{align*}
 B_{1}\leq
 C \|m\|^{1/2}_\infty \|f_2\|_2 r_1^{-|\alpha|}V^{(1)}({x}_1,r_1)^{-1/2}.
 \end{align*}
Similarly,
\begin{align*}
B_{2}
\leq C \|m\|^{1/2}_\infty \|g_2\|_2 r_1^{-|\beta|}V^{(1)}({y}_1,r_1)^{-1/2}.
\end{align*}
Therefore
\begin{align*}
|B|  & \lesssim C \|m\|_\infty\|f_2\|_2\|g_2\|_2\frac{(r_1 \vee \rho^{(1)}(x_1, y_1))^{-|\alpha| -|\beta|}}{V^{(1)}(x_1, \rho^{(1)}(x_1,y_1) +r_1)},
 \end{align*}
as required.
\end{proof}

\begin{proposition} \label{ker}
Suppose that $\sigma_1, \sigma_2 \in \R $ and $m:[0,\infty) \times [0,\infty) \to \C $ has support in $[1/4, 4] \times [1/4, 4]$ and
$\|m\|_{W^{(\sigma_1,\sigma_2),2}(\R \times \R )}<\infty$.
Then, for any $\varepsilon>0$, any $\alpha, \beta, \alpha', \beta' \in \mathcal{I}$,
and any $N_1, N_2 \in \N$ with $N_1 > \frac{Q_1}{2} + (|\alpha| \vee |\beta|)$ and
$N_2 > \frac{Q_2}{2} +  (|\alpha'| \vee |\beta'|)$,
there exists a constant $C = C(\varepsilon, \alpha, \beta, \alpha', \beta', \sigma_1, \sigma_2, N_1, N_2)$
such that
\begin{align*}
&\left|(X_{\alpha}^{(1)})_{x_1} (X_{\alpha'}^{(2)})_{x_2} (X_{\beta}^{(1)})_{y_1} (X_{\beta'}^{(2)})_{y_2} K_{m(r_1^2 \Box_b^{(1)}, r_2^2 \Box_b^{(2)})}(x_1, x_2,y_1, y_2)\right| \\
&\qquad\leq C \|m\|_{W^{(\sigma_1, \sigma_2),2}(\R \times \R )}
 \left(1+\frac{\rho^{(1)}(x_1,y_1)}{r_1} \right)^{N_1 +\varepsilon + {1/2} -\sigma_1}\frac{(r_1 \vee \rho^{(1)}(x_1, y_1))^{-|\alpha| -|\beta|}}{V^{(1)}(x_1, \rho^{(1)}(x_1,y_1) +r_1)}\\
&\hspace{5cm}  \
\times \left(1+\frac{\rho^{(2)}(x_2,y_2)}{r_2} \right)^{N_2 +\varepsilon + {1/2} -\sigma_2}\frac{(r_2 \vee \rho^{(2)}(x_2, y_2))^{-|\alpha'| -|\beta'|}}{V^{(2)}(x_2, \rho^{(2)}(x_2,y_2) +r_2)}.
\end{align*}
\end{proposition}

\begin{proof}
Fix $\bar{x}_1,\bar{x}_2 \in M^{(1)}$ and $\bar{y}_1, \bar{y}_2 \in M_2$.
We wish to bound
\begin{align*}
(X_{\alpha}^{(1)})_{x_1} (X_{\alpha'}^{(2)})_{x_2} (X_{\beta}^{(1)})_{y_1} (X_{\beta'}^{(2)})_{y_2}K_{m(r_1^2 \Box_b^{(1)}, r_2^2\Box_b^{(2)})}(x_1,x_2,y_1,y_2)\Bigm|_{(x_1,x_2,y_1, y_2)=(\bar{x}_1,
\bar{x}_2, \bar{y}_1, \bar{y}_2)}.
\end{align*}
We divide the argument into several cases:

\subsubsection*{Case 1.} $\rho^{(1)}(\bar{x}_1, \bar{y}_1) \geq r_1$ and $\rho^{(2)}(\bar{x}_2, \bar{y}_2) \geq r_2$.
Let $\psi(\lambda_1, \lambda_2) := m(\lambda_1^2, \lambda_2^2)$, $\psi_{r_1, r_2}(\lambda_1, \lambda_2):=\psi(r_1 \lambda_1,
r_2 \lambda_2)$ and let $F_{s_1,s_2,r_1,r_2}(\lambda_1, \lambda_2)$ be given via the relation
\begin{align*}
\MF{ F_{s_1, s_2, r_1,r_2}} (\xi_1, \xi_2) =\phi\left(\frac{\xi_1}{s_1}\right)
 \phi\left(\frac{\xi_2}{s_2}\right) \frac{1}{r_1 r_2}\MF{\psi} \left(\frac{\xi_1}{r_1}, \frac{\xi_2}{r_2}\right)
  = \phi\left(\frac{\xi_1}{s_1}\right)
 \phi\left(\frac{\xi_2}{s_2}\right) \MF{\psi_{r_1,r_2}} (\xi_1, \xi_2),
\end{align*}
where $\phi$ is as in Lemma \ref{elem}.
Let $J_{s_1,s_2,r_1,r_2} (\lambda_1, \lambda_2)$ be a complex-valued function such that
\begin{align*}
J_{s_1,s_2,r_1,r_2} (\lambda_1, \lambda_2)^2= F_{s_1,s_2,r_1,r_2} (\lambda_1, \lambda_2).
\end{align*}

Set $s_1= \frac{\rho^{(1)}(\bar{x}_1,\bar{y}_1)}{4\kappa_1}$ and $s_2= \frac{\rho^{(2)}(\bar{x}_2, \bar{y}_2)}{4\kappa_2}$.
Since
\begin{align*}
\supp \MF{F_{s_1,s_2,r_1,r_2}}(\cdot, \cdot) - \supp \MF{\psi_{r_1, r_2}}(\cdot, \cdot) &\subset [0,2s_1]\times [0,2s_2] \\
&= \left[0,\frac{\rho^{(1)}(\bar{x}_1,\bar{y}_1)}{2\kappa_1}\right] \times \left[0,\frac{\rho^{(2)}(\bar{x}_2, \bar{y}_2)}{2\kappa_2}\right],
\end{align*}
it follows from the finite speed propagation (Proposition \ref{fsp}) that
\begin{align*}
&\supp  \left(K_{F_{s_1,s_2,r_1,r_2}(\sqrt{\Box_b^{(1)}}, \sqrt{\Box_b^{(2)}})}(\cdot, \cdot, \cdot, \cdot) - K_{\psi_{r_1, r_2}(\sqrt{\Box_b^{(1)}}, \sqrt{\Box_b^{(2)}})}(\cdot, \cdot,\cdot,\cdot)\right)\\
&\qquad\subset \left\{(x_1, x_2, y_1, y_2):  \rho^{(1)}(x,y) \leq \frac{\rho^{(1)}(\bar{x}_1, \bar{x}_2)}{2} \mbox{ and } \rho^{(2)}(x_2,y_2)
 \leq \frac{\rho^{(2)}(\bar{x}_2, \bar{x}_2)}{2}\right\}.
\end{align*}
Hence we can write
\begin{align*}
&(X_{\alpha}^{(1)})_{x_1} (X_{\alpha'}^{(2)})_{x_2} (X_{\beta}^{(1)})_{y_1} (X_{\beta'}^{(2)})_{y_2}K_{m(r_1^2\Box_b^{(1)}, r_2^2\Box_b^{(2)})}(x_1,x_2,y_1,y_2)\Bigm|_{(\bar{x}_1,
\bar{x}_2, \bar{y}_1, \bar{y}_2)} \\
&\qquad=(X_{\alpha}^{(1)})_{x_1} (X_{\alpha'}^{(2)})_{x_2} (X_{\beta}^{(1)})_{y_1} (X_{\beta'}^{(2)})_{y_2}K_{\psi_{r_1,r_2}(\sqrt{\Box_b^{(1)}}, \sqrt{\Box_b^{(2)}})}(x_1,x_2,y_1,y_2)\Bigm|_{(\bar{x}_1,
\bar{x}_2, \bar{y}_1, \bar{y}_2)} \\
&\qquad=(X_{\alpha}^{(1)})_{x_1} (X_{\alpha'}^{(2)})_{x_2} (X_{\beta}^{(1)})_{y_1} (X_{\beta'}^{(2)})_{y_2}K_{F_{s_1,s_2,r_1,r_2}(\sqrt{\Box_b^{(1)}}, \sqrt{\Box_b^{(2)}})}(x_1,x_2,y_1,y_2)\Bigm|_{(\bar{x}_1,
\bar{x}_2, \bar{y}_1, \bar{y}_2)} \\
&\qquad=(X_{\alpha}^{(1)})_{x_1} (X_{\alpha'}^{(2)})_{x_2} (X_{\beta}^{(1)})_{y_1} (X_{\beta'}^{(2)})_{y_2}K_{F_{s_1,s_2,r_1,r_2}(\sqrt{\hatbox_b^{(1)}}, \sqrt{\hatbox_b^{(2)}})}(x_1,x_2,y_1,y_2)\Bigm|_{(\bar{x}_1,
\bar{x}_2, \bar{y}_1, \bar{y}_2)} \\
&\qquad\qquad + (X_{\alpha}^{(1)})_{x_1} (X_{\alpha'}^{(2)})_{x_2} (X_{\beta}^{(1)})_{y_1} (X_{\beta'}^{(2)})_{y_2} K_{\pi^{(1)} \otimes F_{s_1,s_2,r_1,r_2} (0, \sqrt{\hatbox_b^{(2)}})} (x_1, y_1, x_2, y_2)\Bigm|_{(\bar{x}_1,
\bar{x}_2, \bar{y}_1, \bar{y}_2)}\\
&\qquad\qquad +(X_{\alpha}^{(1)})_{x_1} (X_{\alpha'}^{(2)})_{x_2} (X_{\beta}^{(1)})_{y_1} (X_{\beta'}^{(2)})_{y_2} K_{F_{s_1,s_2,r_1,r_2} (\sqrt{\hatbox_b^{(1)}}, 0) \otimes \pi^{(2)}} (x_1, y_1, x_2, y_2)\Bigm|_{(\bar{x}_1,
\bar{x}_2, \bar{y}_1, \bar{y}_2)} \\
&\qquad\qquad +F_{s_1,s_2,r_1,r_2} (0, 0) (X_{\alpha}^{(1)})_{x_1} (X_{\alpha'}^{(2)})_{x_2} (X_{\beta}^{(1)})_{y_1} (X_{\beta'}^{(2)})_{y_2} K_{\pi^{(1)} \otimes \pi^{(2)}} (x_1, y_1, x_2, y_2)\Bigm|_{(\bar{x}_1,
\bar{x}_2, \bar{y}_1, \bar{y}_2)} \\
&\qquad:= I_1 + I_2 + I_3 + I_4.
\end{align*}

We first estimate $I_1$. Using the fact that
\begin{align*}
F_{s_1,s_2,r_1,r_2}(\sqrt{\hatbox_b^{(1)}},\sqrt{\hatbox_b^{(2)}}) = J_{s_1,s_2,r_1,r_2}(\sqrt{\hatbox_b^{(1)}},
\sqrt{\hatbox_b^{(2)}}) \circ J_{s_1,s_2,r_1,r_2}(\sqrt{\hatbox_b^{(1)}},\sqrt{\hatbox_b^{(2)}})
\end{align*}
 and the Cauchy--Schwarz inequality, we see that
\begin{align*}
|I_1|
&=\left|(X_{\alpha}^{(1)})_{x_1} (X_{\alpha'}^{(2)})_{x_2} (X_{\beta}^{(1)})_{y_1} (X_{\beta'}^{(2)})_{y_2}K_{F_{s_1,s_2,r_1,r_2}
(\sqrt{\hatbox_b^{(1)}}, \sqrt{\hatbox_b^{(2)}})}(x_1,x_2,y_1,y_2)\Bigm|_{(\bar{x}_1,
\bar{x}_2, \bar{y}_1, \bar{y}_2)}\right|\\
&=  \left|(X_{\alpha}^{(1)})_{x_1} (X_{\alpha'}^{(2)})_{x_2} (X_{\beta}^{(1)})_{y_1} (X_{\beta'}^{(2)})_{y_2}
\int_{M^{(1)} \times M^{(2)}} K_{J_{s_1,s_2,r_1,r_2}
(\sqrt{\hatbox_b^{(1)}}, \sqrt{\hatbox_b^{(2)}})}(x_1,x_2,z_1,z_2)\right.\\
&\hspace{4cm}
 \times \left. K_{J_{s_1, s_2,r_1,r_2}(\sqrt{\hatbox_b^{(1)}}, \sqrt{\hatbox_b^{(2)})}}(z_1,z_2,y_1,y_2)
\,d\mu^1 (z_1)  \,d\mu^2(z_2)\Bigm|_{(\bar{x}_1,
\bar{x}_2, \bar{y}_1, \bar{y}_2)}\right|\\
&\leq \left\|(X_{\alpha}^{(1)})_{x_1} (X_{\alpha'}^{(2)})_{x_2}
 K_{J_{s_1,s_2,r_1,r_2} (\sqrt{\hatbox_b^{(1)}}, \sqrt{\hatbox_b^{(2)}})}(x_1,x_2,
 \cdot,\cdot)\Bigm|_{(x_1, x_2)=(\bar{x}_{1}, \bar{x}_{2})}  \right\|_{L^{2}(M^{(1)} \times M^{(2)})} \\
&\qquad \times \left\|  (X_{\beta}^{(1)})_{y_1} (X_{\beta'}^{(2)})_{y_2}K_{J_{s_1,s_2,r_1,r_2} (\sqrt{\hatbox_b^{(1)}}, \sqrt{\hatbox_b^{(2)}})}(\cdot,\cdot,y_1,y_2)\Bigm|_{(y_1, y_2)
 =(\bar{y}_1, \bar{y}_2)} \right\|_{L^2 (M^{(1)} \times M^{(2)})}\\
 & =:I_{11} \times I_{22}.
\end{align*}
We write
\begin{align*}
&(X_{\alpha}^{(1)})_{x_1} (X_{\alpha'}^{(2)})_{x_2}
 K_{J_{s_1,s_2,r_1,r_2} (\sqrt{\hatbox_b^{(1)}}, \sqrt{\hatbox_b^{(2)}})}(x_1,x_2,
 \cdot,\cdot)\\
 &=G(\sqrt{\Box_b^{(1)}}, \sqrt{\Box_b^{(2)}})(X_{\alpha}^{(1)})_{x_1} (X_{\alpha'}^{(2)})_{x_2}
 K_{\left(I + s_1 \sqrt{\Box_b^{(1)}}\right)^{-N_1 /2}  \otimes \left(I + s_2\sqrt{\Box_b^{(1)}}\right)^{-N_2 /2} }(x_1,x_2,
 \cdot,\cdot),
\end{align*}
where
\begin{align*}
G(\lambda_1, \lambda_2) = (1 + s_1 \lambda_1)^{N_1 /2} (1 + s_2 \lambda_2)^{N_2 /2}J_{s_1, s_2, r_1, r_2}(\lambda_1, \lambda_2).
\end{align*}
By Lemma \ref{elem} (iv),
\begin{align*}
|G(\lambda_1, \lambda_2)| \lesssim \left( 1+ \frac{s_1}{r_1} \right)^{({N_1/2} + {1/4}+{\varepsilon/2})-{\sigma_1/2}}
\left( 1+ \frac{s_2}{r_2} \right)^{({N_2/2} + {1/4} + {\varepsilon/2})-{\sigma_2/2}} .
\end{align*}
From this, Lemma \ref{bes} and Remark \ref{rmk}, it follows that
\begin{align*}
I_{11}
& \lesssim  \left( 1+ \frac{s_1}{r_1} \right)^{({N_1/2} + {1/4}+{\varepsilon/2})-{\sigma_1/2}}
\left( 1+ \frac{s_2}{r_2} \right)^{({N_2/2} + {1/4}+ {\varepsilon/2})-{\sigma_2/2}}
\left\|k \right\|_{L^{2}(M^{(1)} \times M^{(2)})}\\
 & \lesssim  \left( 1+ \frac{s_1}{r_1} \right)^{({N_1/2} + {1/4}+{\varepsilon/2})-{\sigma_1/2}}
\left( 1+ \frac{s_2}{r_2} \right)^{({N_2/2} + {1/4} + {\varepsilon/2})-{\sigma_2/2}} \frac{s_1^{-|\alpha|}}{V^{(1)}(\bar{x}_1,s_1)^{1/2}} \frac{s_2^{-|\alpha'|}}{V^{(2)}(\bar{x}_2, s_2)^{1/2}},
\end{align*}
where
\[
k(\cdot,\cdot) = (X_{\alpha}^{(1)})_{x_1} (X_{\alpha'}^{(2)})_{x_2}
 K_{\left(I + s_1 \sqrt{\Box_b^{(1)}}\right)^{-N_1 /2}  \otimes \left(I + s_2\sqrt{\Box_b^{(1)}}\right)^{-N_2 /2} }(x_1,x_2,
 \cdot,\cdot)\Bigm|_{(x_1, x_2)=(\bar{x}_{1}, \bar{x}_{2})} ;
\]
here we used the assumptions that $N_1 >  Q_1 + 2(|\alpha| \vee |\beta|)$ and  $N_2 >  Q_2 + 2(|\alpha'| \vee |\beta'|)$.
Similarly,
\begin{align*}
I_{12} \lesssim  \left( 1+ \frac{s_1}{r_1} \right)^{({N_1/2} + {1/4}+{\varepsilon/2})-{\sigma_1/2}}
\left( 1+ \frac{s_2}{r_2} \right)^{({N_2/2} + {1/4}+ {\varepsilon/2})-{\sigma_2/2}} \frac{s_1^{-|\beta|}}{V^{(1)}(\bar{y}_1,s_1)^{1/2}} \frac{s_2^{-|\beta'|}}{V^{(2)}(\bar{y}_2, s_2)^{1/2}}.
\end{align*}
Therefore,
\begin{align*}
|I_1| = I_{11} \times I_{12}  & \lesssim \left( 1+\frac{\rho^{(1)}(\bar{x}_1,\bar{y}_1)}
{r_1}\right)^{({N_1} + \varepsilon + {1/2} )- \sigma_1}
  \left( 1+\frac{\rho^{(2)}(\bar{x}_2,\bar{y}_2)}{r_2}\right)^{({N_2} + \varepsilon + {1/2})-{\sigma_2}} \\
&\quad \times \frac{s_1^{-|\alpha|}}{V^{(1)}(\bar{x}_1,s_1)^{1/2}} \frac{s_2^{-|\alpha'|}}{V^{(2)}(\bar{x}_2, s_2)^{1/2}}
\frac{s_1^{-|\beta|}}{V^{(1)}(\bar{y}_1,s_1)^{1/2}} \frac{s_2^{-|\beta'|}}{V^{(2)}(\bar{y}_2, s_2)^{1/2}},
\end{align*}

Next we estimate $I_2$. By Lemma \ref{elem} (iv),
\begin{align*}
|J_{s_1, s_2, r_1,r_2} (0, \lambda_2)| \lesssim
\left( 1+ \frac{s_1}{r_1} \right)^{({N_1/2} + {1/4}+ {\varepsilon/2})- {\sigma_1/2}}
\left( 1+ \frac{s_2}{r_2} \right)^{({N_2/2} + {1/4}+ {\varepsilon/2})- {\sigma_2/2}}
(1+\lambda_2 s_2)^{-{N_2/2}}.
\end{align*}
Hence
\begin{align*}
&\left|(X_{\alpha'}^{(2)})_{x_2} (X_{\beta'}^{(2)})_{y_2} K_{F_{s_1,s_2,r_1,r_2}
(0, \sqrt{\hatbox_b^{(2)}})}(x_2,y_2)\Bigm|_{(\bar{x}_2, \bar{y}_2)}\right|\\
&\qquad=  \left|(X_{\alpha'}^{(2)})_{x_2} (X_{\beta'}^{(2)})_{y_2}
\int_{M^{(2)}} K_{J_{s_1,s_2,r_1,r_2}(0, \sqrt{\hatbox_b^{(2)}})}(x_2,z_2) K_{J_{s_1, s_2,r_1,r_2}(0, \sqrt{\hatbox_b^{(2)}})}(z_2,y_2)\,dz_2\Bigm|_{(\bar{x}_2, \bar{y}_2)} \right|\\
&\qquad \leq \left\|(X_{\alpha'}^{(2)})_{x_2}
 K_{J_{s_1,s_2,r_1,r_2} (0, \sqrt{\hatbox_b^{(2)}})}(x_2,
  \cdot)\Bigm|_{x_2= \bar{x}_2}  \right\|_{L^{2}( M^{(2)})}
\\ &\hspace{4cm}
\left\|(X_{\beta'}^{(2)})_{y_2} K_{J_{s_1,s_2,r_1,r_2} (0,
 \sqrt{\hatbox_b^{(2)}})}(\cdot,y_2)\Bigm|_{y_2 = \bar{y}_2} \right\|_{L^2 ( M^{(2)})}\\
&\qquad\lesssim \left( 1+\frac{s_1}{r_1}\right)^{({N_1/2} + {1/4}+{\varepsilon/2})-{\sigma_1/2}}
\left( 1+\frac{s_2}{r_2}\right)^{({N_2/2} + {1/4}+{\varepsilon/2})-{\sigma_2/2}}
\\
&\hspace{4cm}
  \left\|(X_{\alpha'}^{(2)})_{x_2} K_{(I+s_2 \sqrt{\hatbox_{b}^{(2)}})^{-N_2/2}}
  (x_2, \cdot)\Bigm|_{x_2 =\bar{x}_2}\right\|_{L^2 (M^{(2)})}\\
&\qquad\qquad   \times
  \left( 1+\frac{s_1}{r_1}\right)^{({N_1/2} + {1/4}+{\varepsilon/2})-{\sigma_1/2}}
  \left( 1+\frac{s_2}{r_2}\right)^{({N_2/2} + {1/4}+{\varepsilon/2})-{\sigma_2/2}}
\\
&\hspace{4cm}
  \left\|(X_{\beta'}^{(2)})_{y_2} K_{(I+s_2 \sqrt{\hatbox_{b}^{(2)}})
  ^{-N_2/2}}(\cdot, y_2)\Bigm|_{y_2 =\bar{y}_2}\right\|_{L^2 (M^{(2)})}\\
&\qquad\lesssim  \left( 1+\frac{s_1}{r_1}\right)^{(N_1
  + \varepsilon + {1/2})-\sigma_1}
  \left( 1+\frac{s_2}{r_2}\right)^{({N_2/2} + {1/4}+{\varepsilon/2})-{\sigma_2/2}}
      \frac{s_2^{-|\alpha|}}{V^{(2)} (\bar{x}_2, s_2)^{1/2}} \frac{s_2^{-|\beta|}}{V^{(2)} (\bar{y}_2, s_2)^{1/2}}\\
&\qquad\sim \left( 1+\frac{\rho^{(1)}(\bar{x}_1,\bar{y}_1)}{r_1}\right)^{({N_1} + \varepsilon + {1/2})-{\sigma_1}}\left( 1+\frac{\rho^{(2)}(\bar{x}_2,\bar{y}_2)}{r_2}\right)^{({N_2} + \varepsilon + {1/2})-{\sigma_2}} \\
&\hspace{4cm}\frac{\rho^{(2)}(\bar{x}_2, \bar{y}_2)^{-|\alpha'| -|\beta'|}}{V^{(2)}(\bar{x}_2, \rho^{(2)}(\bar{x}_2,
    \bar{y}_2) +r_2)},
\end{align*}
where we used the assumption that $N_2 >  Q_2 + 2(|\alpha'| \vee |\beta'|)$.

Since $\pi^{(1)}$ is an NIS operator,
\begin{align*}
\left|(X_{\alpha}^{(1)})_{x_1}(X_{\beta}^{(1)})_{y_1} K_{\pi^{(1)}}(x_1,y_1)\Bigg|_{(\bar{x}_1, \bar{y}_1))}\right|
&\leq C_{\alpha, \beta} \frac{\rho^{(1)}(\bar{x}_1,\bar{y}_1)^{-|\alpha|-|\beta|}}{V^{(1)}(\bar{x}_1, \rho^{(1)}(\bar{x}_1,\bar{y}_1))} \\
&\sim \frac{\rho^{(1)}(\bar{x}_1,\bar{y}_1)^{-|\alpha|-|\beta|}}{V^{(1)}(\bar{x}_1, \rho^{(1)}(\bar{x}_1,\bar{y}_1)+r_1)} .
\end{align*}
Therefore,
\begin{align*}
|I_2| &\lesssim \left( 1+\frac{\rho^{(1)}(\bar{x}_1,\bar{y}_1)}{r_1}\right)^{({N_1} +
\varepsilon + {1/2})-{\sigma_1}}   \left( 1+\frac{\rho^{(2)}(\bar{x}_2,\bar{y}_2)}{r_2}\right)^{({N_2} + \varepsilon + {1/2})-{\sigma_2}}\\
&\qquad \times  \frac{\rho^{(1)}(\bar{x}_1,\bar{y}_1)^{-|\alpha|-|\beta|}}{V^{(1)}(\bar{x}_1,
\rho(\bar{x}_1,\bar{y}_1)+r_1)}  \frac{\rho^{(2)}(\bar{x}_2, \bar{y}_2)^{-|\alpha'| -|\beta'|}}
{V^{(2)}(\bar{x}_2, \rho^{(2)}(\bar{x}_2,\bar{y}_2) +r_2)}.
\end{align*}
The estimate of $I_3$ is similar to that of $I_2$.

Finally we estimate $I_4$.
By Lemma 3.2 (iv),
\begin{align*}
|F_{s_1,s_2,r_1,r_2}(0,0)| \lesssim \left( 1+ \frac{s_1}{r_1} \right)^{(N_1 + \varepsilon + {1/2})-\sigma_1}
\left( 1+ \frac{s_2}{r_2} \right)^{(N_2 + \varepsilon + {1/2})-\sigma_2},
\end{align*}
and
\begin{align*}
|I_4|
&\lesssim \left( 1+ \frac{s_1}{r_1} \right)^{(N_1 + \varepsilon + {1/2})-\sigma_1} \left( 1+ \frac{s_2}{r_2} \right)^{(N_2 + \varepsilon + {1/2})-\sigma_2} \\
&\hspace{5cm}
\times \frac{\rho^{(1)}(x_1,y_1)^{-|\alpha|-|\beta|}}{V^{(1)}(x_1, \rho^{(1)}(x_1,y_1))}
\frac{\rho^{(2)}(x_2,y_2)^{-|\alpha'|-|\beta'|}}{V^{(2)}(x_2, \rho^{(2)}(x_2,y_2))}\\
&\sim \left( 1+\frac{\rho^{(1)}(x_1,y_1)}{r_1}\right)^{({N_1} + \varepsilon + {1/2})-{\sigma_1}}   \left( 1+\frac{\rho^{(2)}(x_2,y_2)}{r_2}\right)^{({N_2} + \varepsilon + {1/2})-{\sigma_2}}\\
&\hspace{5cm}
\times  \frac{\rho^{(1)}(x_1,y_1)^{-|\alpha|-|\beta|}}{V^{(1)}(x_1, \rho^{(1)}(x_1,y_1)+r_1)}  \frac{\rho^{(2)}(x_2, y_2)^{-|\alpha'| -|\beta'|}}{V^{(2)}(x_2, \rho^{(2)}(x_2,y_2) +r_2)}.
\end{align*}

\subsubsection*{Case 2.} $\rho^{(1)}(\bar{x}_1, \bar{y}_1) \geq r_1$ and $\rho^{(2)}(\bar{x}_2, \bar{y}_2) \leq r_2$.
Let $\psi(\lambda_1, \lambda_2):= m(\lambda_1^2, \lambda_2^2)$ and define  $F_{s}(\lambda_1, \lambda_2)$ by the formula
\begin{align*}
\MF_1{ F_{s_1, r_1,r_2}} (\xi_1, \lambda_2) = \phi\left(\frac{\xi_1}{s_1}\right) \frac{1}{r_1}
\MF_1{\psi}\left(\frac{\xi_1}{r_1}, r_2 \lambda\right).
\end{align*}
Let $J_{s_1,r_1,r_2} (\lambda_1, \lambda_2)$ be a complex-valued function such that
\begin{align*}
J_{s_1,r_1,r_2} (\lambda_1, \lambda_2)^2 = F_{s_1,r_1,r_2} (\lambda_1, \lambda_2).
\end{align*}
Set $s_1= {\rho^{(1)}(x_1,y_1)/(4\kappa_1)}$.
Then by the finite speed propagation (as in Case 1),
\begin{equation*}
\begin{aligned}
&(X_{\alpha}^{(1)})_{x_1} (X_{\alpha'}^{(2)})_{x_2} (X_{\beta}^{(1)})_{y_1} (X_{\beta'}^{(2)})_{y_2}K_{m(r_1^2\Box_b^{(1)}, r_2^2\Box_b^{(2)})}(x_1,x_2,y_1,y_2)\Bigm|_{(\bar{x}_1,
\bar{x}_2, \bar{y}_1, \bar{y}_2)} \\
&\qquad=(X_{\alpha}^{(1)})_{x_1} (X_{\alpha'}^{(2)})_{x_2} (X_{\beta}^{(1)})_{y_1} (X_{\beta'}^{(2)})_{y_2}K_{F_{s_1,r_1,r_2}(\sqrt{\Box_b^{(1)}}, \sqrt{\Box_b^{(2)}}})(x_1,x_2,y_1,y_2)\Bigm|_{(\bar{x}_1,
\bar{x}_2, \bar{y}_1, \bar{y}_2)} \\
&\qquad=(X_{\alpha}^{(1)})_{x_1} (X_{\alpha'}^{(2)})_{x_2} (X_{\beta}^{(1)})_{y_1} (X_{\beta'}^{(2)})_{y_2}K_{F_{s_1,r_1,r_2}(\sqrt{\hatbox_b^{(1)}}, \sqrt{\hatbox_b^{(2)}})}(x_1,x_2,y_1,y_2)\Bigm|_{(\bar{x}_1,
\bar{x}_2, \bar{y}_1, \bar{y}_2)}\\
&\qquad\qquad + (X_{\alpha}^{(1)})_{x_1} (X_{\alpha'}^{(2)})_{x_2} (X_{\beta}^{(1)})_{y_1} (X_{\beta'}^{(2)})_{y_2}
K_{\pi^1 \otimes F_{s_1,r_1,r_2} (0, \sqrt{\hatbox_b^{(2)}}}) (x_1, y_1, x_2, y_2)\Bigm|_{(\bar{x}_1,
\bar{x}_2, \bar{y}_1, \bar{y}_2)} \\
&\qquad= : I_1 + I_2.
\end{aligned}
\end{equation*}

We first estimate $I_1$.
By Lemma \ref{elem} (ii),
\begin{align*}
|F_{s_1, r_1,r_2} (\lambda_1,\lambda_2)|\lesssim (1 +s_1\lambda_1 )^{-N_1}\left( 1+ \frac{s_1}{r_1} \right)^{(N_1 + \varepsilon + {1/2})-\sigma_1}
(1 +r_2\lambda_2 )^{-N_2}.
\end{align*}
Hence
\begin{align*}
\left|I_1\right|
&= \biggl| (X_{\alpha}^{(1)})_{x_1} (X_{\alpha'}^{(2)})_{x_2} (X_{\beta}^{(1)})_{y_1} (X_{\beta'}^{(2)})_{y_2} \\
&\qquad\int_{M_1 \times M_2} K_{J_{s_1,r_1,r_2} (\sqrt{\hatbox_b^{(1)}}, \hatbox_b^{(2)})}(x_1,x_2,z_1,z_2)
K_{J_{s_1,r_1,r_2} (\sqrt{\hatbox_b^{(1)}}, \hatbox_b^{(2)})}(z_1,z_2,y_1,y_2)
\,dz_1 \,dz_2\biggr|\\
&\leq \left\|(X_{\alpha}^{(1)})_{x_1} (X_{\alpha'}^{(2)})_{x_2}
 K_{J_{s_1,r_1,r_2} (\sqrt{\hatbox_b^{(1)}}, \hatbox_b^{(2)})}(x_1,x_2,
 \cdot,\cdot)\Bigm|_{(x_1, x_2)=(\bar{x}_1, \bar{x}_2)}  \right\|_{L^{2}(M_1 \times M_2)} \\
&\qquad \times \left\|(X_{\beta}^{(1)})_{y_1} (X_{\beta'}^{(2)})_{y_2} K_{J_{s_1,r_1,r_2} (\sqrt{\hatbox_b^{(1)}}, \hatbox_b^{(2)})}(\cdot,\cdot,y_1,y_2)\Bigm|_{(y_1, y_2)
 =(\bar{y}_1, \bar{y}_2)} \right\|_{L^2 (M_1 \times M_2)}\\
&\lesssim \left( 1+\frac{s_1}{r_1}\right)^{({N_1/2} + {1/4}+{\varepsilon/2})-{\sigma_1/2}}
\left( 1+\frac{s_1}{r_1}\right)^{({N_1/2} + {1/4}+{\varepsilon/2})-{\sigma_1/2}}
\\
&\qquad\times \left\|  (X_{\alpha}^{(1)})_{x_1}  K_{(I+s_1 \sqrt{\hatbox_{b}^{(1)}})^{-N_1 /2}}(x_1, \cdot)\Bigm|_{x_1 =\bar{x}_1}\right\|_{L^2 (M^{(1)})}\\
&\qquad\times   \left\| (X_{\alpha'}^{(2)})_{x_2} K_{(I+r_2 \sqrt{\hatbox_{b}^{(2)}})^{-N_2 /2}}(x_2, \cdot)\Bigm|_{x_2 =\bar{x}_{2}}\right\|_{L^2 (M^{(2)})}  \\
&\qquad  \times
 \left\|(X_{\beta}^{(1)})_{y_1} K_{(I+s_1 \sqrt{\hatbox_{b}^{(1)}})^{-N_1/2}}(\cdot, y_1)\Bigm|_{y_1 =\bar{y}_1}\right\|_{L^2 (M^{(1)})}\\
&\qquad \times
  \left\| (X_{\beta'}^{(2)})_{y_2} K_{(I+r_2 \sqrt{\hatbox_{b}^
  {(2)}})^{-N_2 /2}}(\cdot, y_2)\Bigm|_{y_2 =\bar{y}_2}\right\|_{L^2 (M^{(2)})} \\
&\lesssim \left( 1+\frac{\rho^{(1)}(\bar{x}_1,\bar{y}_1)}{r_1}\right)^{({N_1}
   + \varepsilon + {1/2})-{\sigma_1}}
  \frac{s_1^{-|\alpha| -|\beta|}}{V^{(1)}(\bar{x}_1, s_1)} \frac{ r_2^{-|\alpha| -|\beta|}}{V^{(2)}(\bar{x}_2, r_2)}\\
  & \sim \left( 1+\frac{\rho^{(1)}(\bar{x}_1,\bar{y}_1)}{r_1}\right)^{({N_1} + \varepsilon + {1/2})-{\sigma_1}}
  \left( 1+\frac{\rho^{(2)}(\bar{x}_2, \bar{y}_2)}{r_2}\right)^{({N_2} + \varepsilon + {1/2})-{\sigma_2}}\\
&\qquad \times  \frac{\rho^{(1)}(\bar{x}_1,\bar{y}_1)^{-|\alpha| -|\beta|}}{V^{(1)}(\bar{x}_1, \rho^{(1)}(\bar{x}_1,\bar{y}_1)+r_1)} \frac{ \rho^{(2)}(\bar{x}_2,\bar{y}_2)^{-|\alpha'| -|\beta'|}}
{V^{(2)}(\bar{x}_2, \rho^{(2)}(\bar{x}_2,\bar{y}_2)+r_2)}
\end{align*}
since $\rho^{(2)}(\bar{x}_2, \bar{y}_2) \leq r_2$.

\subsubsection*{Case 3.} $\rho^{(1)}(\bar{x}_1, \bar{y}_1) \leq r_1$ and $\rho^{(2)}(\bar{x}_2, \bar{y}_2) \geq r_2$.
This case is very similar to Case 2, except that we use Lemma 3.1 (iii) instead of Lemma 3.1 (ii).

\subsubsection*{Case 4.} $\rho^{(1)}(\bar{x}_1, \bar{y}_1) \leq r_1$ and $\rho^{(2)}(\bar{x}_2, \bar{y}_2) \leq r_2$.
This case can be easily handled by using Lemma 3.1 (i). We omit the details.
\end{proof}

\section{Proof of Theorem \ref{main}}

Our aim is to show that $m(\Box_b^{(1)}, \Box_b^{(2)})$ is bounded on $L^p (M^{(1)} \times M^{(2)})$. To do this, we will show that
$m(\Box_b^{(1)}, \Box_b^{(2)})$ is a product singular integral operator satisfying all the conditions in Proposition \ref{mmm}.

First, we write
\begin{equation} \label{sisa}
\begin{split}
m(\Box_b^{(1)}, \Box_b^{(2)}) &=[ m(0,0) \pi^{(2)} \otimes \pi^{(2)}]  +  \big[\pi^{(1)} \otimes m^2(\hatbox_b^{(2)})\big]\\
&\qquad + \big[m^1(\hatbox_b^{(1)}) \otimes \pi^{(2)}\big] + \big[m(\hatbox_b^{(1)}, \hatbox_b^{(2)})\big],
\end{split}
\end{equation}
where
\begin{align*}
m^1(\lambda_1) := m(\lambda_1, 0) \quad \mbox{and} \quad m^2(\lambda_2):=m(0, \lambda_2).
\end{align*}

Since $\pi^{(k)}$ is an NIS operator on $M^{(k)}$,  it is bounded on $L^{2}(M^{(k)})$ for $1<p <\infty$ and satisfies
(i) and (ii) of Definition~\ref{NIS} for $k=1,2$.
Hence $\pi^{(1)} \otimes \pi^{(2)}$ satisfies all the conditions in Proposition \ref{mmm}.
Moreover, since $m$ satisfies \eqref{con2} and \eqref{con3}, it follows from Street's result on one-parameter spectral multipliers
associated with Kohn Laplacians that $m^1(\hatbox_b^{(1)})$ and $m^2(\hatbox_b^{(2)})$ satisfy
(i) and (ii) of Definition~\ref{NIS}.
This, together with the fact that $\pi^{(1)}$ and $\pi^{(2)}$ are bounded on $L^2$, implies that both $\pi^{(1)} \otimes m^2(\hatbox_b^{(2)})$ and $m^1(\hatbox_b^{(1)}) \otimes \pi^{(2)}$
satisfy all the conditions in Proposition \ref{mmm}.
Thus, it remains to show that $m(\hatbox_b^{(1)}, \hatbox_b^{(2)})$ satisfies the conditions in Proposition \ref{mmm}.

Let $\omega$ be a nonnegative smooth function on $[0, \infty)$ supported in
$[0,2]$ such that $\omega(\lambda) =1$ for $\lambda \in [0,1]$.
Let $\theta(\lambda) = \omega(\lambda) -\omega(2\lambda)$ for all $\lambda \in [0,\infty)$.
Then $\supp \theta \subset [1/2,2]$ and
\begin{align*}
\sum_{j =-\infty}^{+\infty} \theta(2^{-j}\lambda) &= \sum_{j=-\infty}^{+\infty}
\big[\omega(2^{-j}\lambda) - \omega(2^{-(j-1)}\lambda)  \big]\\
&= \sum_{j =-\infty}^{-1} \big[\omega(2^{-j}\lambda) - \omega(2^{-(j-1)}\lambda)  \big]
 + \sum_{j=0}^{+\infty}\big[\omega(2^{-j}\lambda) - \omega(2^{-(j-1)}\lambda)  \big]\\
&= \lim_{j \to -\infty}\big[\omega(2\lambda) -\omega(2^{-j}\lambda) \big] + \lim_{j \to +\infty}
 \big[\omega(2^{-j}\lambda)- \omega(2\lambda)\big]\\
&= -\lim_{j \to -\infty}\omega(2^{-j}\lambda)  + \lim_{j \to +\infty} \omega(2^{-j}\lambda)\\
&=0 + 1=1.
\end{align*}
Set $m_{j_1,j_2} ({\Box}_b^{(1)}, {\Box}_b^{(2)})= \theta(2^{-j_1}{\Box}_b^{(1)})\theta(2^{-j_2}{\Box}_b^{(2)})m({\Box}_b^{(1)}, {\Box}_b^{(2)})$.
Now $m_{j_1,j_2}(0,0)=0$, and it follows that $m_{j_1,j_2} ({\Box}_b^{(1)}, {\Box}_b^{(2)})=m_{j_1,j_2}(\hatbox_b^{(1)}, \hatbox_b^{(2)})$.
This implies that
\[
m(\hatbox_b^{(1)}, \hatbox_b^{(2)})=\sum_{j_1,j_2} m_{j_1,j_2} ({\Box}_b^{(1)}, {\Box}_b^{(2)}).
\]

Similarly, define
\[
m_{j_1} ({\Box}_b^{(1)}, \hatbox_b^{(2)})= \theta(2^{-j_1}{\Box}_b^{(1)})m({\Box}_b^{(1)}, \hatbox_b^{(2)})
\]
and
\[
m_{j_2} ({\Box}_b^{(1)}, \hatbox_b^{(2)})= \theta(2^{-j_2}{\Box}_b^{(2)})m(\hatbox_b^{(1)}, {\Box}_b^{(2)}).
\]
Then
\[
m(\hatbox_b^{(1)}, \hatbox_b^{(2)})=\sum_{j_1} m_{j_1} ({\Box}_b^{(1)}, \hatbox_b^{(2)})
\]
and
\[
m(\hatbox_b^{(1)}, \hatbox_b^{(2)})=\sum_{j_2} m_{j_2} (\hatbox_b^{(1)},{\Box}_b^{(2)}).
\]
Next we need the following auxiliary lemma.

\begin{lemma}\label{summation}
Let $\delta=\rho(x,y)$ and $N > a\geq 0$.
\begin{equation}
\sum_{j=-\infty}^{\infty} \left(1+\frac{\delta}{2^{-j}}\right)^{-N}\frac{2^{ja}}{V(x,\delta+2^{-j})}\leq C\frac{\delta^{-a}}{V(x,\delta)}.
\end{equation}
\end{lemma}

\begin{proof}
The volume of the ball satisfies the reverse doubling condition~\eqref{reverse doubling} and thus
\begin{align*}
&\sum_{2^{-j}\geq \delta} \left(1+\frac{\delta}{2^{-j}}\right)^{-N}\frac{2^{ja}}{V(x,\delta+2^{-j})}\\
&\qquad\leq C_N \sum_{2^{-j}\geq \delta} \left(\frac{\delta}{2^{-j}}\right)^4\frac{2^{ja}}{V(x,\delta)}\\
&\qquad\leq C_N \frac{\delta^{-a}}{V(x,\delta)}.
\end{align*}
On the other hand, since $N>a$,
\begin{align*}
&\sum_{2^{-j}< \delta} \left(1+\frac{\delta}{2^{-j}}\right)^{-N}\frac{2^{ja}}{V(x,\delta+2^{-j})}\\
&\qquad\leq C_N \sum_{2^{-j}< \delta} \left(\frac{\delta}{2^{-j}}\right)^{-N}\frac{2^{ja}}{V(x,\delta)}\\
&\qquad\leq C_N \frac{\delta^{-a}}{V(x,\delta)}.
\end{align*}

The proof of Lemma \ref{summation} is complete.
\end{proof}

\begin{proof}[\bf Proof of Theorem~\ref{main}]
We only need to check that $K_{m(\hatbox_b^{(1)}, \hatbox_b^{(2)})}$ satisfies the conditions of Proposition~\ref{mmm}.
Define $m_{j_1,j_2}, m_{j_1}$ and $m_{j_2}$ as above.
Then from Proposition~\ref{ker} and Lemma~\ref{summation}, it follows that
\begin{align*}
&\left| (X_{\alpha}^{(1)})_{x_1} (X_{\beta}^{(1)})_{y_1} (X_{\alpha'}^{(2)})_{x_2}
 (Y^{(2)}_{\beta'})_{y_2}K_{m(\hatbox_b^{(1)}, \hatbox_b^{(2)})}(x_1, x_2, y_1, y_2)\right|\\
&\qquad\lesssim \sum_{j_1,j_2 } \left |(X_{\alpha}^{(1)})_{x_1} (X_{\beta}^{(1)})_{y_1} (X_{\alpha'}^{(2)})_{x_2}
(Y^{(2)}_{\beta'})_{y_2}m_{j_1,j_2} ({\Box}_b^{(1)}, {\Box}_b^{(2)})(x_1, x_2, y_1, y_2)\right|\\
&\qquad\leq \sum_{j_1,j_2} \|\theta(\cdot_1)\theta(\cdot_2)m(2^{j_1}\cdot_1, 2^{j_2}\cdot_2)\|_{W^{(\sigma_1, \sigma_2),2}(\R \times \R )} \\
&\qquad\qquad \times \left(1+\frac{\rho^{(1)}(x_1,y_1)}{2^{-j_1}} \right)^{N_1 +\varepsilon + {1/2} -\sigma_1}\frac{(2^{-j_1} \vee \rho^{(1)}(x_1, y_1))^{-|\alpha| -|\beta|}}{V^{(1)}(x_1, \rho^{(1)}(x_1,y_1) +2^{-j_1})}\\
&\qquad\qquad
\times \left(1+\frac{\rho^{(2)}(x_2,y_2)}{2^{-j_2}} \right)^{N_2 +\varepsilon + {1/2} -\sigma_2}\frac{(2^{-j_2} \vee \rho^{(2)}(x_2, y_2))^{-|\alpha'| -|\beta'|}}{V^{(2)}(x_2, \rho^{(2)}(x_2,y_2) +2^{-j_2})}\\
&\qquad\leq  C\sup_{t_1, t_2 >0} \|\eta (\lambda_1) \eta (\lambda_2)
m(t_1 \lambda_1, t_2 \lambda_2)\|_{W^{(\sigma_1, \sigma_2), 2}} \frac{\rho^{(1)} (x_1, y_1)^{-|\alpha| - |\beta|} \rho^{(2)}(x_2, y_2)^{-|\alpha'|-|\beta'|}}{V^{(1)}(x_1, y_1)V^{(2)}(x_2,y_2)}.
\end{align*}

Similarly, from Proposition~\ref{ker1} and Lemma~\ref{summation} it follows that
\begin{align*}
&\left|\int_{M^{(2)}}\int_{M^{(2)}}(X_{\alpha}^{(1)})_{x_1}(X_{\beta}^{(1)})_{y_1}K_{m(\hatbox_b^{(1)}, \hatbox_b^{(2)})}(x_1, x_2,y_1, y_2)f_2(x_2)g_2(y_2)\,d\mu^{(2)}(x_2)\,d\mu^{(2)}(y_2) \right|\\
&\qquad\leq \sum_{j_1}\biggl|\int_{M^{(2)}}\int_{M^{(2)}}(X_{\alpha}^{(1)})_{x_1}(X_{\beta}^{(1)})_{y_1}K_{m_{j_1} ({\Box}_b^{(1)}, \hatbox_b^{(2)})}(x_1, x_2,y_1, y_2) \\
&\hspace{4cm} f_2(x_2)g_2(y_2)\,d\mu^{(2)}(x_2)\,d\mu^{(2)}(y_2) \biggr|\\
&\qquad\leq C \sum_{j_1}\|\theta(\cdot_1)m(2^{j_1}\cdot_1, \cdot_2)\|_{W^{\sigma_1,2}L^\infty(\R \times \R )}\|f_2\|_2\|g_2\|_2 \\
&\hspace{4cm} \times \left(1+\frac{\rho^{(1)}(x_1,y_1)}{2^{-j_1}} \right)^{N_1 +\varepsilon + {1/2} -\sigma_1}\frac{(2^{-j_1} \vee \rho^{(1)}(x_1, y_1))^{-|\alpha| -|\beta|}}{V^{(1)}(x_1, \rho^{(1)}(x_1,y_1) +2^{-j_1})}\\
&\qquad\leq C\sup_{t_1, t_2 >0} \|\eta (\lambda_1) \eta (\lambda_2)
m(t_1 \lambda_1, t_2 \lambda_2)\|_{W^{(\sigma_1, \sigma_2), 2}}
\|f_2\|_2\|g_2\|_2\frac{\rho^{(1)}(x_1,y_1)^{-|\alpha|-|\beta|}}{V^{(1)}(x_1, y_1)}.
\end{align*}
Note that this estimate still holds if the indices $1$ and $2$ are interchanged, that is, if the roles of $M^{(1)}$ and
$M^{(2)}$ are changed. Moreover, since $m(\Box_b^{(1)}, \Box_b^{(2)})$ is a self-adjoint operator, its transpose also satisfies these conditions. Hence, $m(\Box_b^{(1)}, \Box_b^{(2)})$, defined spectrally on $L^2(\widetilde{M})$, is a product Calder\'on--Zygmund operator on $\widetilde{M}$ as studied in
\cite{HLL}.
\end{proof}

\newpage

\section*{Acknowledgements}
Cowling is supported by the Australia Research Council, through grant DP170103025.
Hu is supported by the National Natural Science Foundation of China  (Grant No. 11901256) and the Natural Science Foundation of Jiangxi Province  (Grant No. 20192BAB211001).
Li is supported by the Australia Research Council, through grant DP170101060.


\begin{thebibliography}{99}

\bibitem{Ch}
L.K. Chen,
The multiplier operators on the product spaces,
\textit{Illinois J. Math.} \textbf{38} (1994), 420--433.


\bibitem{CDLWY}
P. Chen, X. Duong, J.  Li, L.A. Ward and L. Yan,
Marcinkiewicz-type spectral multipliers on Hardy and Lebesgue spaces on product spaces of homogeneous type,
\textit{J. Fourier Anal. Appl.} {\bf23} (2017), 21--64.

\bibitem{CDLWY2}
P. Chen, X. Duong, J. Li,  L.A. Ward and L. Yan,
Product Hardy spaces associated to operators with heat kernel bounds on spaces of homogeneous type,
\textit{Math. Z.} {\bf282} (2016), 1033--1065.


\bibitem{GS}
R. Gundy and E.M. Stein,
$H^p$ theory for the poly-disc,
\textit{Proc. Nat. Acad. Sci.} \textbf{76} (1979), 1026--1029.


\bibitem{HLL}
Y.S. Han, J. Li and C.C. Lin,
Criterions of the $L^2$ boundedness and sharp endpoint estimates for singular integral operators on product spaces of homogeneous type,
\textit{Ann. Scuola Norm. Sup. Pisa Cl. Sci. (5)} \textbf{XVI} (2016), 845--907.


\bibitem{J}
J.-L. Journ\'e,
Calder\'on--Zygmund operators on product spaces,
\textit{Rev.~Mat.~Iberoamericana} \textbf{1} (1985), 55--91.

\bibitem {K}
K.D Koenig,
{On maximal Sobolev and H\"{o}lder estimates for the tangential Cauchy--Riemann operator and boundary Laplacian},
\textit{Amer. J. Math.} \textbf{124} (2002), 129--197.

\bibitem {Mel}
R. Melrose,
Propagation for the wave group of a positive subelliptic second-order differential operator.
\textit{Hyperbolic equations and related topics (Katata/Kyoto, 1984)},
pp. 181--192.
Academic Press, Boston, (1986).


\bibitem {NRSW}
A. Nagel, J.-P. Rosay, E.M. Stein and S. Wainger,
{Estimates for the Bergman and Szeg\"{o}  kernels in $\C ^2$},
\textit{Ann. of Math. (2)} \textbf{129} (1989), 113--149.


\bibitem {NS01}
A. Nagel and E. M. Stein,
{Differentiable control metrics and scaled bump functions},
\textit{J. Differential Geom.} \textbf{57} (2001), 465--492.

\bibitem {NS}
A. Nagel and E. M. Stein,
{The $\Box_b$-heat equation on pseudoconvex manifolds of finite type in $\C ^2$},
\textit{Math. Z.} \textbf{238} (2001), 37--88.



\bibitem {NS1}
A. Nagel and E. M. Stein,
{On the product theory of singular integrals},
\textit{Rev. Mat. Iberoamericana} \textbf{20} (2004), 531--561.


\bibitem {NS2}
A. Nagel and E. M. Stein,
{The $\bar{\partial}_b$-complex on decoupled boundaries in $\C^n$},
\textit{Ann. Math.} {\bf164} (2006), 649--713.

\bibitem {NSW}
A. Nagel, E. M. Stein and S. Wainger,
{Balls and metrics defined by vector fields. I. Basic properties},
\textit{Acta Math.} \textbf{155} (1985), 103--147.

\bibitem {Str1}
B. Street,
{The $\Box_b$-heat equation and multipliers via the wave equation},
\textit{Math. Z.} \textbf{263} (2009), 861--886.

\bibitem {Str2}
B. Street,
\textit{Multi-paramter singular integrals.}
Annals of Mathematics Studies, 189.
Princeton University Press, Princeton, NJ, 2014.
\end{thebibliography}
\end{document}